\documentclass[12pt]{amsart}

\usepackage{amsmath,amsfonts,latexsym,graphicx,amssymb,amsthm}
\usepackage[tips,matrix,arrow]{xy}

\usepackage{xcolor}

\newcommand{\C}{\mathbb{C}}

\newcommand{\R}{\mathbb{R}}
\newcommand{\Z}{\mathbb{Z}}

\newcommand{\K}{\mathbb{K}}

\newcommand{\nh}{\mathbb{H}}

\newcommand{\fb}{\mathfrak{b}}

\newcommand{\fg}{\mathfrak{g}}
\newcommand{\fh}{\mathfrak{h}}

\newcommand{\fn}{\mathfrak{n}}

\newcommand{\fS}{\mathfrak{S}}

\newcommand{\vep}{\varepsilon}

\newcommand{\ch}{\mathrm{ch}}

\newcommand{\End}{\mathrm{End}}

\newcommand{\rank}{\mathrm{rank}}

\newcommand{\Ind}{\mathrm{Ind}}

\newcommand{\mult}{\mathrm{mult}}

\newcommand{\tr}{\mathrm{tr}}

\newcommand{\cE}{\mathcal{E}}

\newcommand{\cK}{\mathcal{K}}

\newcommand{\cO}{\mathcal{O}}
\newcommand{\cP}{\mathcal{P}}

\newcommand{\tTh}{\widetilde{Th}}

\newtheorem{conj}{Conjecture}[section]
\newtheorem{cor}{Corollary}[section]
\newtheorem{thm}{Theorem}[section]
\newtheorem{prop}{Proposition}[section]
\newtheorem{lemma}{Lemma}[section]
\newtheorem{rem}{Remark}[section]
\newtheorem{Def}{Definition}[section]

\newenvironment{Ac}%
 {\hspace*{-1.2em}\textbf{Acknowledgment.}\hspace{0.7em}}{}

\begin{document}

\title{ Invariants of the Weyl group of type $A_{2l}^{(2)}$}
\author{Kenji IOHARA and Yosihisa SAITO}
\address{Univ Lyon, Universit\'{e} Claude Bernard Lyon 1, CNRS UMR 5208, Institut Camille Jordan, 
43 Boulevard du 11 Novembre 1918, F-69622 Villeurbanne cedex, France}
\email{iohara@math.univ-lyon1.fr}
\address{Departement of Mathematics, Rikkyo University, Toshima-ku, Tokyo 171-8501, Japan}
\email{yoshihisa@rikkyo.ac.jp,\, yosihisa@ms.u-tokyo.ac.jp}

\maketitle

\begin{center}
\textit{Dedicated to Professor Kyoji Saito on the occasion of his 75th Anniversary}
\end{center}

\begin{abstract} 
In this note, we show the polynomiality of the ring of invariants with respect to the Weyl group of type $A_{2l}^{(2)}$. 
\end{abstract}

\tableofcontents

\section*{Introduction}
Let $V$ be a finite dimensional vector space over a field $\K$ of characteristic $0$. 
A \textit{reflection} is a finite order linear transformation on $V$ that fixes a 
hyperplane. Such a linear transformation is called \textit{real} if it is of order $2$ and 
\textit{complex} if it is of order $>2$. 
A group $G$ of linear transformations is a \textit{finite reflection group} if it is a 
finite group generated by reflections. In particular, a finite reflection group generated 
only by real reflections are called \textit{Coxeter group}, otherwise it is called a 
\textit{complex refelction group}. The $G$-action on $V$ induces an $G$-action on 
its dual $V^\ast$, hence on its symmetric algebra $S^\bullet(V^\ast)$. C. Chevalley 
(for real case) \cite{Ch} and
G. C. Shephard and J. A. Todd (for complex case)  \cite{ST} have shown that the ring 
of invariants $S(V^\ast)^G$ is a $\Z$-graded algebra generated by $\dim_K V$ 
homogeneous elements that are algebraically independent. Notice that, for a finite 
group, there is the so-called Reynold's operator to obtain invariants explicitly. 

Now, for an affine Weyl group $W$, one can consider the $W$-invariant theta 
functions defined on a half space $Y$ of a Cartan subalgebra $\fh$ of the affine Lie 
algebra $\fg$ whose Weyl group is $W$. 
A first attempt has been made to determine the structure of the ring of such 
$W$-invariant theta functions by E. Looijenga \cite{L} in 1976 whose argument 
has not been sufficient, as was pointed out by I. N. Bernstein and O. Schwarzman 
\cite{BS} in 1978. Indeed, the ring $\tTh^+$ of $W$-invariant theta functions is an 
$\cO_\nh$-module, where $\nh$ is the Poincar\'{e} upper half plane 
$\{ \tau \in \C \vert \, \mathrm{Im}\, \tau>0\, \}$, spanned by the normalized 
characters $\{\chi_{\Lambda}\}_{\Lambda \in P_+ \mod \C \delta}$ of simple 
integrable highest weight modules over $\fg$. We remark that each 
$\chi_\Lambda$ is holomorphic on $Y$, as was shown by M. Gorelik and V. Kac 
\cite{GK}. 

In 1984, V. G. Kac and D. Peterson \cite{KP} has published a long monumental article 
on the modular transformations, in particular, the Jacobi transformation, of the 
characters of simple integrable highest weight modules over affine Lie algebras. In 
particular, as an application of their results, they presented a list of the Jacobian of 
the fundamental characters for affine Lie algebras except for type $F_4^{(1)}, 
E_6^{(1)}, E_6^{(2)}, E_7^{(1)}$ and $E_8^{(1)}$, and stated that the ring $\tTh^+$ of 
$W$-invariant theta functions is a polynomial ring over $\cO_\nh$ generated by 
fundamental characters $\{\chi_{\Lambda_i}\}_{0\leq i\leq l}$. Unfortunately, their 
proofs have never been published (cf. Ref. [35] in \cite{KP}). In 2006, J. Bernstein and 
O. Schwarzman in \cite{BS1} and \cite{BS2} presented a detailed version of their 
announcement \cite{BS}, where their final result is a weak form of what V. Kac and D. 
Peterson has announced. Moreover, J. Bernstein and O. Schwarzman excluded two 
cases: type $D_l^{(1)}$ and $A_{2l}^{(2)}$.  \\

In this article, we determine the Jacobian of the fundamental characters for the 
affine Lie algebra $\fg$ of type $A_{2l}^{(2)}$. In particular, 
we show that the explicit form of Jacobian in Table J of \cite{KP} for type 
$A_{2l}^{(2)}$ is valid. As a corollary, it follows that the fundamental characters 
$\{\chi_{\Lambda_i}\}_{0\leq i\leq l}$ of $\fg$ are algebraically 
independent. 

This article is organised as follows.
As the practical computation requires many detailed information, 
Section 1 is devoted to  providing root datum, detailed description on the non-
degenerate symmetric invariant bilinear form restricted on a Cartan subalgebra 
$\fh$ and its induced bilinear
form on the dual $\fh^\ast$, and the structure of the affine Weyl group of type 
$A_{2l}^{(2)}$. In Section 2, we recall the theta functions associated to the 
Heisenberg subgroup of the affine Weyl group of type $A_{2l}^{(2)}$ and the modular 
transformation of the normalized characters of irreducible integrable highest weight 
$\fg$-modules. In Section 3, 
After some technical preliminary computations, we study the modular 
transformations of the Jacobian of the fundamental characters. Finally in Section 4, 
we determine the Jacobian of the fundamental characters of type $A_{2l}^{(2)}$. \\

This text contains many well-known facts about the characters of integrable highest weight modules over affine Lie algebras for the sake of reader's convenience. \\

\begin{Ac} At an early stage of this project, K. I. was partially supported by the program FY2018 JSPS Invitation.  
He is also partially supported by the French ANR (ANR project ANR-15-CE40-0012).
Y .S. is partially supported by JSPS KAKENHI Grant Number 16K05055.
\end{Ac}

\section{Preliminaries}\label{sect_data}
In this section, we recall the definition of the affine Lie algebra of type $A_{2l}^{(2)}$
and its basic properties. 
Everything given in this section is completely known, see e.g., \cite{Kac} and/or 
\cite{MP}. Nevertheless, in order  to fix our convention clearly and to avoid 
unnecessary confusion, we collect several explicit data that would be useful for 
further computations. 
%----------------------------------------------
\subsection{Basic data}\label{sect_data}
Here, we fix the enumeration of the vertices in the Dynkin diagram of type 
$A_{2l}^{(2)}$ as follows (cf. \cite{Kac}): 
\[ \UseTips
\xymatrix @=1.2pc @R=0.8pt
{
&&\alpha_0&\alpha_1&&&\alpha_{l-2}&\alpha_{l-1}&\alpha_{l} \\
&A_{2l}^{(2)}~(l\geq 2) \quad
&{\xy*\cir<4pt>{}\endxy}\ar@2{<-}[r]
&{\xy*\cir<4pt>{}\endxy}\ar@{-}[r]
&\ar@{.}[r]
&\ar@{-}[r]
&{\xy*\cir<4pt>{}\endxy}\ar@{-}[r]
&{\xy*\cir<4pt>{}\endxy}\ar@2{<-}[r]
&{\xy*\cir<4pt>{}\endxy} \\
&&&&&&&&\\
&&&&&&&& \\
&&\alpha_0&\alpha_1&&&&& \\
&A_2^{(2)} \qquad
&**[r] {\xy*\cir<4pt>{}\endxy}
  < \ar@{=} @<2pt> [r] \ar@{=} @<-2pt> [r]
&{\xy*\cir<4pt>{}\endxy}
&&&&& \\
}
\]

In particular, this implies that the corresponding generalized Cartan matrix 
$A=(a_{i,j})_{0\leq i,j\leq l}$ is given by
\[ A=\begin{pmatrix}
   2 & -2 & 0 & \cdots &\cdots & \cdots &  0 \\
   -1 & 2 & -1 & \ddots &         &            & \vdots \\
   0  & -1 & 2 & \ddots &\ddots &        & \vdots \\
   \vdots & \ddots & \ddots & \ddots & \ddots &\ddots & \vdots \\
   \vdots && \ddots & \ddots  & 2 & -1 & 0   \\
   \vdots &&& \ddots & -1 & 2 & -2 \\
   0 & \cdots & \cdots & \cdots & 0 & -1 & 2
   \end{pmatrix}
\]
for $l\geq 2$, and
\[ A=\begin{pmatrix} 2 & -4 \\ -1 & 2 \end{pmatrix}
\]
for $l=1$. A realization of the 
generalized Cartan matrix $A$ is a triple $(\fh, \Pi, \Pi^\vee)$ such that 
\begin{itemize}
\item[(i)] $\fh$ is an $l+2$-dimensional $\C$-vector space, 
\item[(ii)] $\Pi^\vee=\{h_i\}_{0\leq i\leq l}$ is a linearly independent subset of $\fh$, 
\item[(iii)] $\Pi=\{\alpha_i\}_{0\leq i\leq l}$ is a linearly independent subset of 
$\fh^\ast:=\mbox{Hom}_{\mathbb{C}}(\fh,\mathbb{C})$, and
\item[(iv)] $\langle h_i, \alpha_j \rangle=a_{i,j}$ for any 
$0\leq i,j\leq l$.
\end{itemize}
Here, $\langle \cdot,\cdot\rangle:\fh\times \fh^*\to \mathbb{C}$ 
is the canonical pairing. 

The labels $a_0,a_1,\cdots, a_l$ and co-labels $a_0^\vee, a_1^\vee, \cdots, 
a_l^\vee$ are by definition, relatively prime positive integers satisfying
\[ (a_0^\vee, a_1^\vee, \cdots, a_l^\vee)A=(0,0,\cdots, 0), \qquad
    A\begin{pmatrix} a_0 \\ a_1 \\ \vdots \\ a_l \end{pmatrix}=
    \begin{pmatrix} 0 \\ 0 \\ \vdots \\ 0 \end{pmatrix}.
\]
Explicitly, they are given by the following table: 
 \begin{center}
    \begin{tabular}{|c||c|c|c|c|c|} \hline
  
  $i$ & $0$ & $1$ & $\cdots $ & $l-1$ & $l$ \\ \hline
    $a_i$         & $2$ & $2$ & $\cdots$ & $2$ & $1$  \\ \hline
    $a_i^\vee$ & $1$ & $2$ & $\cdots$ & $2$ & $2$ \\ \hline  
    \end{tabular}
\end{center}

As a corollary, the Coxeter number $h$ and the dual Coxeter number $h^\vee$ are 
shown to be
\[ h:=\sum_{i=0}^l a_i=2l+1, \qquad h^\vee:=\sum_{i=0}^l a_i^\vee=2l+1. \]

Let us introduce some special elements $c$ of 
$\fh$, and $\delta$ of $\fh^\ast$ by
\[c:=\sum_{i=0}^l a_i^\vee h_i=h_0+2(h_1+\cdots +h_l) \in \fh, \qquad 
\delta:=\sum_{i=0}^l a_i\alpha_i=2(\alpha_0+\cdots+\alpha_{l-1})+
\alpha_l. \] 
By definition, we have
\[ \langle c, \alpha_i\rangle =0, \qquad \langle h_i, \delta \rangle=0, \]
for any $0\leq i\leq l$. 
Now, define $d \in \fh$ and $\Lambda_0 \in \fh^\ast$
the elements characterized by
\[ \langle d, \alpha_i\rangle =\delta_{i,0}, \qquad 
   \langle h_i, \Lambda_0 \rangle=\delta_{i,0}, \qquad
   \langle d, \Lambda_0 \rangle =0. \]
It can be checked that
\[ \langle c, \Lambda_0 \rangle=a_0^\vee=1, \qquad \langle d, \delta 
\rangle=a_0=2. \]

For later purpose, we recall the definition of $i$-th fundamental weight 
$\Lambda_i \in \fh^\ast$, the element of $\fh^\ast$ satisfying
\[ \langle h_i, \Lambda_j\rangle=\delta_{i,j} \qquad (0\leq i,j\leq l), \qquad \langle d, 
\Lambda_i \rangle=0.\]
Clearly, one has
\[ \langle c, \Lambda_i \rangle=a_i^\vee.\]
%--------------------------------------------------
\subsection{Bilinear forms on $\fh$ and $\fh^\ast$}\label{sect_normalized-form}
Here, we recall the bilinear form $I$ on $\fh^\ast$ and $I^\ast$ on $\fh$ that are 
called normalized invariant forms by V. Kac \cite{Kac} in an explicit manner. 

On $\fh$:
\begin{align*}
&I^\ast(h_i, h_j)=a_j(a_j^\vee)^{-1}a_{i,j} \qquad (0\leq i,j\leq l), \\
&I^\ast(h_i, d)=0 \qquad (0<i\leq l), \\
&I^\ast(h_0, d)=a_0=2, \qquad I^\ast(d,d)=0.
\end{align*}
In particular, we have
\[ I^\ast(c,d)=2. \]

On $\fh^\ast$: 
\begin{align*}
&I(\alpha_i,\alpha_j)=a_i^\vee (a_i)^{-1}a_{i,j} \qquad (0\leq i,j\leq l), \\
&I(\alpha_i,\Lambda_0)=0 \qquad (0<i\leq l), \\
&I(\alpha_0, \Lambda_0)=a_0^{-1}=\frac{1}{2}, \qquad 
I(\Lambda_0, \Lambda_0)=0.
\end{align*}
We also have 
\[ I(\Lambda_0,\delta)=1.\]
 
The above bilinear forms are introduced in such a way that the linear map 
$\nu: (\fh, I^\ast) \rightarrow (\fh^\ast, I)$ defined by
 \[  a_i^\vee\nu(h_i)=a_i\alpha_i, \qquad \nu(c)=\delta, \qquad 
 \nu(d)=a_0\Lambda_0=2\Lambda_0\]
 becomes an isometry. 
%-------------------------------------
\subsection{Affine Lie algebra of type $A_{2l}^{(2)}$}
Let $\fg=\fg(A)$ be the Kac-Moody Lie algebra attached with the generalized 
Cartan matrix $A$ which is introduced in Subsection \ref{sect_data}. Since
$A$ is symmetrizable, it is defined to be the Lie algebra over the complex number
field $\mathbb{C}$ generated by $e_i,f_i\ (0\leq i\leq l)$ and $\fh$, with 
the following defining relations{\rm :} 
\begin{itemize}
\item[(i)] $[e_i,f_j]=\delta_{i,j}h_i$ \quad for $0\leq i,j\leq l$,
\vskip 1mm
\item[(ii)] $[h,h']=0$ \quad for $h,h'\in\fh$, 
\vskip 1mm
\item[(iii)] $[h,e_i]=\langle h,\alpha_i\rangle e_i$,\quad 
$[h,f_i]=-\langle h,\alpha_i\rangle f_i$\quad for $h\in\fh$ and $0\leq i\leq l$,
\vskip 1mm
\item[(iv)] $(\mbox{ad}\, e_i)^{-a_{i,j}+1}(e_j)=0$,\quad  
$(\mbox{ad}\, f_i)^{-a_{i,j}+1}(f_j)=0$\quad for $0\leq i\ne j\leq l$.
\vskip 1mm
\end{itemize}
The Lie algebra $\fg$ is called the
\textbf{affine Lie algebra of type $A_{2l}^{(2)}$}.
The vector space $\fh$ is regarded as a commutative subalgebra of $\fg$, 
which is called the Cartan subalgebra of $\fg$. \\

It is well-known that $\fg$ admits the root space decomposition:
\[\fg=\fh\oplus\Bigl(\bigoplus_{\alpha\in \Delta}\fg_{\alpha}\Bigr),\]
where $\fg_{\alpha}:=\{X\in\fg\,|\,[h,X]=\langle h,\alpha\rangle X\mbox{ for every }
h\in \fh\}$ is the root space associated to $\alpha\in\fh^*$, 
and $\Delta:=\{\alpha\in\fh^*\setminus\{0\}\,|\,\fg_{\alpha}\ne\{0\}\}$
is the set of all roots. An element $\alpha\in\Delta$ is called a real root if 
$I(\alpha,\alpha)>0$, and the set of all real roots is denoted by $\Delta^{re}$. 
Set $\Delta^{im}:=\Delta\setminus \Delta^{re}$. An element $\alpha\in\Delta^{im}$
is called an imaginary root. The explicit forms of real and imaginary roots will be
given in the next subsection. 

For $\alpha\in \Delta^{re}$, define a reflection $r_{\alpha}\in GL(\fh^*)$ by
$$r_{\alpha}(\lambda):=\lambda-\left\langle \dfrac{2}{I(\alpha,\alpha)}
\nu^{-1}(\alpha),\lambda\right\rangle\alpha$$ 
for $\lambda\in\fh^*$. Let $W$ be the Weyl group of $\fg$. That is, $W$ is
a subgroup of $GL(\fh^*)$ generated by $r_{\alpha}\ (\alpha\in\Delta^{re})$. 
It is known that the action of $W$ on $\fh^*$ preserves the set of roots $\Delta$. 
The detailed structure of the Weyl group $W$ will be given in the next subsection
also. 

%------------------------------------------
\subsection{Affine Weyl group}\label{sect_affine-Weyl}
Here, we describe the structure of the Weyl group of type $A_{2l}^{(2)}$, paying 
attention to the fact that the lattice of the translation part is \textbf{not} an even 
lattice. 

To describe its structure, we regard the $0$-th node as the special vertex, namely, 
the complementary subdiagrams 
\[ \UseTips
\xymatrix @=1.2pc @R=0.7pt 
 {
&&\alpha_1&\alpha_2&&&\alpha_{l-2}&\alpha_{l-1}&\alpha_l \\
&C_l~(l\geq 2) \quad
&{\xy*\cir<4pt>{}\endxy}\ar@{-}[r]
&{\xy*\cir<4pt>{}\endxy}\ar@{-}[r]
&\ar@{.}[r]
&\ar@{-}[r]  
&{\xy*\cir<4pt>{}\endxy}\ar@{-}[r]
&{\xy*\cir<4pt>{}\endxy}\ar@2{<-}[r]
&{\xy*\cir<4pt>{}\endxy}  \\
&&&&&&&& \\
&&&&&&&& \\
&& \alpha_1 &&&&&& \\%\fh
&A_1 \qquad 
&{\xy*\cir<4pt>{}\endxy}
&&&&&& \\
}\]
will be called the finite part of the Dynkin diagram of type $A_{2l}^{(2)}$. \\

Let $\fh_f$ be the subspace of $\fh$ defined by 
$$\fh_f:=\bigoplus_{i=1}^l \C h_i.$$
This can be viewed as the Cartan subalgebra of the 
finite part. In particular, we fix the splitting of $\fh$ as follows:
\begin{equation}\label{eq_splitting1}
 \fh=\fh_f \oplus (\C d \oplus \C c). 
\end{equation}
This is an orthogonal decomposition with respect to the normalized invariant form 
introduced in the previous subsection. We denote the canonical projection to the 
first component by $\overline{\cdot}: \fh \twoheadrightarrow \fh_f$. 
Similarly, we set 
$$\fh_f^\ast:=\bigoplus_{i=1}^l \C \alpha_i$$ 
and the canonical projection to the first component of the decomposition
\begin{equation}\label{eq_splitting2}
 \fh^\ast=\fh_f^\ast \oplus (\C \delta \oplus \C \Lambda_0) 
\end{equation}
will be denoted by the same symbol $\overline{\cdot}: \fh^\ast 
\twoheadrightarrow \fh_f^\ast.$ It should be noticed that the image of the 
restriction of $\nu: \fh \rightarrow \fh^\ast$ to $\fh_f$ is $\fh_f^\ast$.\\

Now, we recall the root system of the finite part. Although most of the informations 
can be found in \cite{Bour} for example, we describe this to fix the convention. 

We identify $\fh_f$ and $\fh_f^\ast$ via the isometry $\nu$.  An orthonormal basis 
$\{\vep_i\}_{1\leq i\leq l}$ of $\fh_f^\ast$ with respect to $I$
should be so chosen that the root system $\Delta_f$ of the finite part and 
the set 
of simple roots $\Pi_f=\{\alpha_i\}_{1\leq i\leq l}$ has the next description: 
\[ \Delta_f=\{\pm \vep_i\pm \vep_j\}_{1\leq i,j\leq l} \cup 
\{\pm 2\vep_i\}_{1\leq i\leq l}, \qquad
\alpha_i=\begin{cases}  \vep_i-\vep_{i+1} \qquad & i<l, \\ 2\vep_i \qquad & i=l.
\end{cases}
\]
\noindent{\underline{\textbf{N.B.}}}\quad For $l=1$, $\Delta_f=\{\pm 2\vep_1\}$. 
\qquad $\Box$ \\

In particular, the sets $\Delta_f^+$, $\Delta_{f,l}^+$ and $\Delta_{f,s}^+$ of positive 
roots, positive long roots and short roots, respectively are given by
\[ \Delta_f^+=\{\vep_i\pm \vep_j\}_{1\leq i<j\leq l} \cup \{2\vep_i\}_{1\leq i\leq l}, 
\qquad
\Delta_{f,l}^+=\{2\vep_i\}_{1\leq i\leq l}, \qquad 
\Delta_{f,s}^+=\{\vep_i\pm \vep_j\}_{1\leq i<j\leq l}.
\]
We remark that $\varpi_i:=\overline{\Lambda_i}=\sum_{j=1}^{i}\vep_j$ for $1\leq 
i\leq l$ is the $i$-th fundamental weight of $\fh_f$.
The highest root $\theta \in \fh_f^*$ can be written as
\[ \theta=\delta-a_0\alpha_0=\delta-2\alpha_0=2\vep_1, \]
and its coroot $\theta^\vee\in \fh_f$ as
\[ \theta^\vee=\frac{2}{I(\theta,\theta)}\nu^{-1}(\theta)=\sum_{i=1}^l h_i, 
\qquad \nu(\theta^\vee)=\frac{1}{2}\theta=\vep_1. \]

Let $W_f$ be a subgroup generated by $r_{\alpha_i}$ for
$\alpha_i\in \Pi_f$. This group is isomorphic to the Weyl group of type $C_l$-type. 
That is, $W_f \cong \fS_l \ltimes (\Z/2\Z)^l$. \\

The set of real roots $\Delta^{re}$ of $\fg$ is described as follows: for $l\geq 2$, 
\begin{align*}
\Delta^{re}=
&\Delta^{re}_s \cup \Delta_m^{re} \cup \Delta_l^{re}, \\
&\text{where} \quad \begin{cases}
\Delta^{re}_s:=\left. \left\{ \frac{1}{2}(\alpha+(2r-1)\delta)\, \right\vert \alpha 
\in\Delta_{f,l},\, r \in \Z\, \right\}, & \\
\Delta_m^{re}:=\{\alpha+r\delta \, \vert\,  \alpha \in \Delta_{f,s}, \, r \in \Z\, \}, & \\
\Delta_l^{re}:=\{\alpha+2r\delta \, \vert\,  \alpha \in \Delta_{f,l}, \, r \in \Z\, \},
\end{cases}
\end{align*}
and for $l=1$, 
\begin{align*}
\Delta^{re}=
&\Delta^{re}_s \cup \Delta_l^{re}, \\
&\text{where} \quad \begin{cases}
\Delta^{re}_s:=\left. \left\{ \frac{1}{2}(\alpha+(2r-1)\delta)\, \right\vert \alpha \in\Delta_{f},\, r \in \Z\, \right\}, & \\
\Delta_l^{re}:=\{\alpha+2r\delta \, \vert\,  \alpha \in \Delta_{f}, \, r \in \Z\, \}.
\end{cases}
\end{align*}
The set of imaginary roots can be described as 
\[ \Delta^{im}=\Z \delta \setminus \{0\}.  \]
For detail, see \cite{Kac}. \\

Now, we can describe the Weyl group $W$. As $c=h_0+2\theta^\vee$, one has
\begin{align*}
r_{\alpha_0}r_\theta(\lambda)=
&r_{\alpha_0}(\lambda-\langle \theta^\vee, \lambda\rangle \theta) 
=
\lambda-\lambda(\theta^\vee)\theta-\langle h_0, \lambda-\langle \theta^\vee, 
\lambda\rangle\theta\rangle \alpha_0 \\
=
&\lambda+\frac{1}{2}\langle c, \lambda\rangle\theta-\left(\langle \theta^\vee, 
\lambda\rangle+\frac{1}{2}\langle c, \lambda\rangle \right)\delta \\
=
&\lambda+\langle c, \lambda\rangle\nu(\theta^\vee)-\left(I(\lambda, 
\nu(\theta^\vee))+\frac{1}{2}\langle c, \lambda\rangle I( \nu(\theta^\vee), 
\nu(\theta^\vee))\right)\delta,
\end{align*}
for $\lambda \in \fh^\ast$. 
Set 
$$M:=\Z\nu(W_f\cdot \theta^\vee)=\bigoplus_{i=1}^l \Z \vep_i,$$
and, for $\alpha \in \fh_{f,\R}^\ast:=M\otimes_{\Z} \R$, we define 
$t_\alpha \in \End_\C(\fh^\ast)$ by
\[ t_\alpha(\lambda):=\lambda+\langle c, \lambda \rangle\alpha-\left( I(\lambda,
\alpha)+\frac{1}{2}\langle c, \lambda \rangle I( \alpha, \alpha) \right)\delta. \]

It is known that 
the assignment the lattice $M$ to $\End_\C(\fh^\ast)$ defined by 
$\alpha$ to $t_{\alpha}$ gives an injective group homomorphism 
$M \hookrightarrow W$.
Furthermore, one has
\[ W \cong W_f \ltimes M\cong 
\bigl(\mathfrak{S}_l \ltimes (\Z/2/\Z)^l\bigr)\ltimes M. \]
We note that the lattice $M$ is an \textbf{odd} lattice.
This fact is essential in the following discussion.

%%%%%%%%%%%%%%%%%%%%%%%%%%%%%%%%%%%%%%%%%%%%
\section{Main theorem}
Let $\fg$ be the affine Lie algebra of type $A_{2l}^{(2)}$ recalled in the previous 
section.

\subsection{Formal characters}
Let $\fn_\pm$ be the subalgebra of $\fg$ generated by $\{ e_i\}_{0\leq i\leq l}$ 
(resp. $\{f_i\}_{0\leq i\leq l}$). We have the so-called triangular decomposition: 
$\fg=\fn_+ \oplus \fh \oplus \fn_-$. \\

Let $V$ be a $\fh$-diagonalizable module, i.e., 
$V=\bigoplus_{\lambda \in \fh^\ast} V_\lambda$ where 
$V_\lambda:=\{v \in V \, \vert \, 
h.v=\langle h, \lambda\rangle v \quad h \in \fh \, \}$. 
Set $\cP(V):=\{ \lambda \in \fh^\ast \vert V_\lambda\neq \{0\}\}$.
Let $\cO$ be the BGG category of $\fg$-modules, that is, it is the subcategory of 
$\fg$-modules whose objects are $\fh$-diagonalizable $\fg$-modules 
$V=\bigoplus_{\lambda \in \fh^\ast} V_\lambda$ satisfying
\begin{enumerate}
\item[(i)] $\dim V_\lambda< \infty$ for any $\lambda \in \cP(V)$, 
\item[(ii)] there exists $\lambda_1, \lambda_2, \cdots, \lambda_r \in \fh^\ast$ 
such that $\cP(V) \subset \bigcup_{i=1}^r (\lambda_i -\Z_{\geq 0}\Pi)$.
\end{enumerate}
A typical object of this category is a so-called highest weight module defined as 
follows. We say that a $\fg$-module $V$ is a highest weight module with highest weight 
$\Lambda \in \fh^\ast$ if
\begin{itemize}
\item[(i)] $\dim V_\Lambda=1$ 
\item[(ii)] $\fn_+.V_\Lambda=\{0\}$ and $V=U(\fg).V_\Lambda$.
\end{itemize}
In particular, the last condition implies that $V=U(\fn_-).V_\Lambda$ as a vector 
space and 
\[ \cP(V):=\{ \lambda \in \fh^\ast \vert V_\lambda\neq \{0\}\} \subset 
\Lambda - \Z_{\geq 0}\Pi. \]
A typical example is given as follows. 
For $\Lambda \in \fh^\ast$, let $\C_\Lambda=\C v_\Lambda$ be the one 
dimensional module over $\fb_+:=\fh \oplus \fn_+$ defined by 
\[ h.v_\Lambda=\langle h, \Lambda \rangle v_\lambda \quad (h \in \fh), \qquad 
\fn_+.v_\Lambda=0. \]
The induced $\fg$-module  $M(\Lambda):=\Ind_{\fb_+}^{\fg}\C_\Lambda$ is called 
the Verma module with highest weight $\Lambda$. It can be shown that for any 
highest weight $\fg$-module $V$ with highest weight $\Lambda \in \fh^\ast$, 
there exists a surjective $\fg$-module map $M(\Lambda) \twoheadrightarrow V$. 
The smallest among such $V$ can be obtained by taking the quotient of 
$M(\Lambda)$ by its maximal proper submodule and the resulting $\fg$-module is 
the irreducible highest $\fg$-module with highest weight $\Lambda$, denoted by 
$L(\Lambda)$. 

Let $\cE$ be the formal linear combination of $e^{\lambda}$ 
($\lambda \in \fh^\ast$) with the next condition: 
$\sum_\lambda c_\lambda e^\lambda \in \cE$ $\Rightarrow$ 
$\exists \lambda_1, \lambda_2, \cdots,; \lambda_r \in \fh^\ast$ such that
\[ \{ \lambda \vert c_\lambda \neq 0\} \subset 
\bigcup_{i=1}^r (\lambda_i - \Z_{\geq 0}\Pi). \]
We introduce the ring structure on $\cE$ by 
$e^{\lambda} \cdot e^{\mu}:=e^{\lambda+\mu}$. \\

The formal character of $V \in \cO$ is, by definition, the element $\ch V \in \cE$ 
defined by
\[ \ch V=\sum_{\lambda \in \cP(V)} (\dim V_\lambda)e^\lambda. \]
$\ch (\cdot )$ can be viewed as an additive function defined on $\cO$ with 
values in $\cE$. For example, The formal character of the Verma module 
$M(\Lambda)$ is given by
\[ \ch M(\Lambda)=e^{\Lambda} \prod_{\alpha \in \Delta_+}(1-e^{-\alpha})^{-
\mult(\alpha)}, \]
where $\Delta_+$ is the set of positive roots of $\fg$ and 
$\mult(\alpha)=\dim \fg_\alpha$ is the multiplicity of the root $\alpha$. 
Set $\vep(w)=(-1)^{l(w)}$ where $l(w)$ is the length of an element $w \in W$. 
For $\Lambda \in P_+:=\{ \Lambda \in \fh^\ast \vert \, \langle h_i, \Lambda\rangle 
\in \Z_{\geq 0} \vert 0\leq i\leq l\}$, the character of $L(\Lambda)$ is known 
as Weyl-Kac character formula and is given by
\[ \ch L(\Lambda)=\frac{\sum_{w \in W} \vep(w)e^{w(\Lambda+\rho)-\rho}}
{\prod_{\alpha \in \Delta_+}(1-e^{-\alpha})^{\mult(\alpha)}}, \]
where $\rho \in \fh^\ast$ is the so-called Weyl vector, i.e., it satisfies $\langle h_i, 
\rho\rangle =1$ for any $0\leq i\leq l$ and $\langle d, \rho\rangle=0$. 
In particular, for $\Lambda=0$, as $L(0)=\C$ is the trivial representation, one 
obtains the so-called denominator identity:
\begin{equation}\label{denominator-id}
\sum_{w \in W} \vep(w) e^{w(\rho)}
=e^\rho \prod_{\alpha \in \Delta_+}(1-e^{-\alpha})^{\mult(\alpha)}. 
\end{equation}
This implies that the Weyl-Kac character formula can be rephrased as follows:
\begin{equation}\label{Weyl-Kac}
\ch L(\Lambda)=\frac{\sum_{w \in W} \vep(w)e^{w(\Lambda+\rho)}}
{\sum_{w \in W} \vep(w)e^{w(\rho)}}.
\end{equation}
%----------------------------------------------------------
\subsection{Normalized characters of affine Lie algebra of type $A_{2l}^{(2)}$}
\label{sect_ch}
For $\Lambda \in P_+$ with $\langle c, \Lambda\rangle=k \in \Z_{\geq 0}$, the 
normalized character $\chi_\Lambda$ of the irreducible highest weight $\fg$-
module $L(\Lambda)$ with highest weight $\Lambda$ is defined by 
\[\chi_{\Lambda}:=e^{-m_\Lambda \delta} \ch L(\Lambda),\] 
where the number $m_\Lambda$,  called the conformal anomaly, is defined by 
\[m_\Lambda:=\frac{I(\Lambda+\rho,\Lambda+\rho)}
 {2(k+h^\vee)}-\frac{I(\rho,\rho)}{2h^\vee},\]
with $h^\vee=2l+1$ being the dual Coxeter number of $A_{2l}^{(2)}$. 
Here, the bilinear form is normalized as in \S \ref{sect_normalized-form} (cf. Chapter 
6 of \cite{Kac}). For $\Lambda \in P_+$ with $\langle c, \Lambda\rangle=k$, set
\[ A_{\Lambda+\rho}:=\sum_{w \in W}\vep(w)e^{w(\Lambda+\rho)-
\frac{I(\Lambda+\rho,\Lambda+\rho)}{2(k+h^\vee)}\delta}.
\]
This is a $W$-skew invariant. 
By \eqref{Weyl-Kac}, one has
\begin{equation}\label{Weyl-Kac-A}
\chi_\Lambda=\frac{A_{\Lambda+\rho}}{A_\rho}.
\end{equation}

We regard this normalized character $\chi_\Lambda$ as a function defined on a 
certain subset of $\fh$ as follows. For $\lambda \in \fh^\ast$, 
$e^\lambda$ can be viewed as a function defined on $\fh$: 
$e^{\lambda}(h):=e^{\langle h, \lambda\rangle}$. It is shown in \cite{GK} that the 
normalized character $\chi_\Lambda$ for $\Lambda \in P_+$ can be viewed as a 
holomorphic function on a complex domain
\[ Y:=\{ h \in \fh \vert \mathrm{Re}\, \langle h, \delta \rangle >0\, \}. \]
We introduce a coordinate system on $\fh$ as follows. Set $\epsilon_i:=\nu^{-1}
(\vep_i)$ for $1\leq i\leq l$, where $\nu: \fh \rightarrow \fh^\ast$ is the isometry
defined in \S \ref{sect_normalized-form}. Then, $\{\epsilon_i\}_{1\leq i\leq l}$ 
is an orthonormal basis of $\fh_f^\ast$ with respect to 
$I^\ast\vert_{\fh_f^\ast \times \fh_f^\ast}$. 
For $h \in Y$, define $z \in \fh_f, \tau, t \in \C$ as in \cite{Kac}: 
\begin{equation}\label{coord-system-I}
h=2\pi \sqrt{-1}\left(z-\tau \nu^{-1}(\Lambda_0)+t\nu^{-1}(\delta)\right)
=2\pi\sqrt{-1}\left(z-\frac{1}{2}\tau d+tc\right), 
 \end{equation}
and write $z=\sum_{i=1}^l z_i \epsilon_i$. With this coordinate system, we see that
\begin{equation}\label{coord-system-II}
 Y \overset{\sim}{\longrightarrow} \nh \times \fh_f \times \C \cong 
 \nh \times \C^l \times \C; \quad h \; \longmapsto \;
(\tau, z, t) \, \mapsto \, (\tau, (z_1,z_2,\cdots, z_l), t), 
\end{equation}
where $\nh$ is the upper half plane 
$\{\,\tau \in \C\, \vert \,\mathrm{Im}(\tau)>0\,\}$.  We write 
$\chi_\Lambda(\tau, z,t):=\chi_\Lambda(h)$. \\

 For $k \in \Z_{>0}$, set
\[ P^k:=\{\lambda \in \fh^\ast \vert \, \langle c, \lambda\rangle =k, \,\, 
\overline{\lambda} \in M\},
\]
and for $\lambda \in P^k$, set
\[ \Theta_\lambda:=e^{-\frac{I(\lambda, \lambda)}{2k}\delta}\sum_{\alpha \in M}
e^{t_\alpha(\lambda)}. \]
This is the \textbf{classical theta function of degree $k$}. 
Its value on $h \in \fh$ is given by
\begin{equation}\label{def_theta-function}
\Theta_\lambda(\tau,z,t)=
e^{2\pi\sqrt{-1}kt}\sum_{\gamma \in M+k^{-1}\overline{\lambda}}
e^{\pi \sqrt{-1}k\tau I(\gamma, \gamma)+2\pi\sqrt{-1}kI(\gamma, \nu (z))}.
\end{equation}
Now, we recall the modular transformations of the classical theta
functions. Recall that 
$SL(2,\Z)$ is the group generated by
\[ S:=\begin{pmatrix} 0 & -1 \\ 1 & 0 \end{pmatrix}, \qquad 
T:=\begin{pmatrix} 1 & 1 \\ 0 & 1 \end{pmatrix}. \]
As $SL(2,\Z)$ acts on the complex domain $Y$ by 
\[ \begin{pmatrix} a & b \\ c & d \end{pmatrix}.(\tau, z, t)
=\left(\frac{a\tau+b}{c\tau+d}, \frac{z}{c\tau+d}, t-\frac{cI^\ast(z,z)}{2(c\tau+d)}
\right), \]
This action induces a right action of $SL(2,\Z)$ on $\cO_Y$ 
as follows: for $g \in SL(2,\Z)$ and $F \in \cO_Y$, 
\[ F\vert_g (\tau, z,t):=F(g.(\tau, z, t)). \]
The next formula is a simple application of the Poisson resummation 
formula:
\begin{prop}[cf. Theorem 13.5 in \cite{Kac}] Let $\lambda \in P^k$. One has
\begin{align*}
\Theta_\lambda\left(-\frac{1}{\tau}, \frac{z}{\tau}, t-\frac{I^\ast(z,z)}{2\tau}\right)=
&\left(\frac{\tau}{\sqrt{-1}}\right)^{\frac{1}{2}l}k^{-\frac{1}{2}l}
\sum_{\mu \in P^k \mod \C \delta+kM}e^{-\frac{2\pi\sqrt{-1}I(\overline{\lambda},
\overline{\mu})}{k}}\Theta_\mu(\tau,z,t), \\
\Theta_{\lambda}(\tau+2,z,t)=
&e^{\frac{2\pi\sqrt{-1}}{k} I(\overline{\lambda}, \overline{\lambda})}
\Theta_\lambda(\tau,z,t).
\end{align*}
In particular, when $k$ is even, one also has
\[ \Theta_\lambda(\tau+1,z,t)=
e^{\frac{\pi\sqrt{-1}}{k} I(\overline{\lambda}, \overline{\lambda})}
\Theta_\lambda(\tau,z,t).
\]
\end{prop}
In fact, the lattice $M$ is not an even lattice. \\

For $k \in \Z_{\geq 0}$, let $P_+^k:=P_+ \cap P^k$ be the set of 
dominant integral weights of level $k$.  As an application of the above proposition, 
V. Kac and D. Peterson \cite{KP} proved that the $\C$-span of 
$\{A_{\Lambda+\rho}\}_{\Lambda \in P_+^k \, \text{mod}\, \C \delta}$ admits 
an action of a certain subgroup of $SL(2,\Z)$, so does 
the $\C$-span of the normalized characters 
$\{\chi_\lambda(\tau,z,t)\}_{\lambda \in P_k^+ \mod \C \delta}$.  
\begin{thm}[\cite{KP}]\label{thm_KP} 
Let $\fg$ be the affine Lie algebra of type $A_{2l}^{(2)}$. 
For $\lambda \in P_+^{\,k}$\, $( k \in \Z_{>0} )$, one has
\[ \chi_\lambda\left( -\frac{1}{\tau}, \frac{z}{\tau}, 
t-\frac{I^{\ast}(z,z)}{2\tau}\right)
=\sum_{\mu \in P_+^{\,k}\mod \C \delta} a(\lambda, \mu)\chi_\mu(\tau, z,t), \]
where the matrix $(a(\lambda, \mu))_{\lambda, \mu \in P_+^{\,k}\mod \C \delta}$ 
is given by
\begin{align*}
a(\lambda,\mu) &=
(\sqrt{-1})^{l^2}(k+2l+1)^{-\frac{1}{2}l}\\
&\qquad \times\sum_{w \in \fS_l \ltimes (\Z/2\Z)^l}\vep(w)
\exp\left(-\frac{2\pi\sqrt{-1}
I(\overline{\lambda+\rho}, w(\overline{\mu+\rho}))}{k+2l+1}\right). 
\end{align*}
One also has
\[ \chi_\lambda(\tau+2,z,t)=e^{4\pi\sqrt{-1}m_{\lambda}}
\chi_\lambda(\tau,z,t).
\]
\end{thm}

\begin{rem}\label{rem_modular-denom}
For later purpose, we recall the modular transformations of the denominator: 
\begin{align*}
A_\rho\left(-\frac{1}{\tau}, \frac{z}{\tau},
t-\frac{I^{\ast}(z,z)}{2\tau}\right)= 
&\left(\frac{\tau}{\sqrt{-1}}\right)^{\frac{1}{2}l}(\sqrt{-1})^{-l^2}A_\rho(\tau,z,t), \\
A_\rho(\tau+2,z,t)=
&\exp\left( \frac{l(l+1)}{3}\pi\sqrt{-1}\right)A_\rho(\tau,z,t).
\end{align*}
\end{rem}

\begin{rem} \label{rem_fin-Weyl}
It can be checked that, for $w=\sigma \eta \in W_f$ \; $(\sigma \in \fS_l, 
\eta=(\eta_1, \cdots, \eta_l) \in \{\pm 1\}^l)$,  one has 
$\det{}_{\fh_f^\ast}(w)=\mathrm{sgn}(\sigma)\prod_{i=1}^l \eta_i$. 
\end{rem}

Notice that, for each $k \in \Z_{\geq 0}$,  the $\C$-span of the normalized 
characters $\chi_\lambda\, (\lambda \in P_+^k \mod \C \delta)$ is 
$\Gamma_\theta$-stable,  where $\Gamma_\theta$ is the subgroup of $SL(2,\Z)$ 
generated by $S$ and $T^2$.\\

In the following, we recall an expansion of $\chi_{\Lambda}$ in terms of the 
classical theta functions. An element $\lambda\in \cP(L(\Lambda))$ is called maximal 
if $\lambda+\delta\not\in \cP(L(\Lambda))$, and let $\mbox{max}(\Lambda)$ be the 
set of all maximal elements of $\cP(L(\Lambda))$. Hence, we have a decomposition of
$\cP(L(\Lambda))$:
\[ \cP(L(\Lambda))=\bigsqcup_{\lambda\in \mbox{\scriptsize max}(\Lambda)}
\{\lambda-n\delta\,|\,n\in \Z_{\geq 0}\}.\]
For $\lambda\in \mbox{max}(\Lambda)$, set
\[c^{\Lambda}_{\lambda}:=e^{-m_{\Lambda}(\lambda)\delta}
\sum_{n=0}^{\infty}\bigl(\dim_{\C}L(\Lambda)_{\lambda-n\delta}\bigr)e^{-n\delta},
\quad
\mbox{where }m_{\Lambda}(\lambda):=m_{\Lambda}-\frac{I(\lambda,\lambda)}{2k}.
\]
Furthermore, we extend the definition of $c^{\Lambda}_{\lambda}$ to an arbitrary 
$\lambda\in \fh^{\ast}$ as follows. If 
$(\lambda+\C\delta)\cap \mbox{max}(\Lambda)=\emptyset$, we set 
$c^{\Lambda}_{\lambda}:=0$. Otherwise, there exists a unique 
$\mu\in \mbox{max}(\Lambda)$ such that $\lambda-\mu\in \C{\delta}$. Hence, 
we set $c_{\lambda}^{\Lambda}:=c_{\mu}^{\Lambda}$. 

Similarly to the case of classical theta functions, we regard the series 
$c_{\lambda}^{\Lambda}$ as a (formal) function on $Y$. 
By the definition, this function depends only on the variable 
$\tau\in \mathbb{H}$. The following proposition is well-known in the 
representation theory of affine Lie algebras. For example, see \cite{Kac},
 in detail. 
\begin{prop}[\cite{Kac}]\label{prop:ch}
Let $k>0$ be a positive integer and $\Lambda\in P^k_+$.
\begin{enumerate}
\item[\rm (i)] The series $c_{\lambda}^{\Lambda}=c_{\lambda}^{\Lambda}(\tau)$ 
converges absolutely on the upper half plane $\mathbb{H}$ to a holomorphic 
function.
\item[\rm (ii)] The normalized character $\chi_{\Lambda}$ has the following 
expansion in terms of the classical theta functions{\rm :}
\end{enumerate}
\[
\chi_{\Lambda}=\sum_{\lambda \in P^k\,\mbox{\rm \scriptsize mod}\,(kM+\C\delta)}
c_{\lambda}^{\Lambda}\Theta_{\lambda}. 
\]
\end{prop}
The holomorphic function $c_{\lambda}^{\Lambda}(\tau)$ is called the 
string function of $\lambda\in \fh^{\ast}$. 

%-----------------------------------------------------
\subsection{Ring of Theta functions}
For $\alpha \in \fh_{f,\R}^\ast$, we define $p_\alpha \in \End_\C(\fh^\ast)$ by
\[ p_\alpha(h)=h+2\pi\sqrt{-1}\nu^{-1}(\alpha). \]
Set 
\[ N=\fh_{f,\R}^\ast \times \fh_{f,\R}^\ast \times \sqrt{-1}\R \]
and define a group structure on $N$ by
\[ (\alpha, \beta, u) \cdot (\alpha', \beta', u')=(\alpha+\alpha', \beta+\beta', u+u'+
\pi\sqrt{-1}\{I(\alpha, \beta')-I(\alpha',\beta)\}).
\]
The group $N$ is called \textbf{Heisenberg group}. This group acts on $\fh$ by
\[ (\alpha, \beta, u).h=t_\beta(h)+2\pi\sqrt{-1}\nu^{-1}(\alpha)+\{u-\pi\sqrt{-1}
I(\alpha,\beta)\}c. 
\]
This $N$ action preserves the complex domain $Y$. We note that 
$(\alpha, 0,0).h=p_\alpha(h)$ and $(0,\beta,0).h=t_\beta(h)$. We consider the 
subgroup $N_\Z$ of $N$ generated by $(\alpha, 0, 0), (0, \beta, 0)$ with 
$\alpha, \beta \in M$ and $(0, 0, u)$ with $u \in 2\pi \sqrt{-1}\Z$, i.e., 
\[ N_\Z=\{ (\alpha, \beta, u) \in N \, \vert \, \alpha, \beta \in M, \, u+\pi\sqrt{-1}
I(\alpha, \beta) \in 2\pi\sqrt{-1}\Z\, \}.
\]

Let $\cO_Y$ be the ring of holomorphic functions on $Y$. We define the right 
$N$-action on $\cO_Y$ by
\[ \left. F\right\vert_{(\alpha, \beta, u)}(h)=F((\alpha, \beta, u).h) \qquad 
F \in \cO_Y. \]

\begin{Def}
For $k \in \Z_{\geq 0}$, define
\[ \tTh_k:=\left\{ F \in \cO_Y \left\vert \; \begin{matrix} 
(1) 
&\; \left. F\right\vert_{(\alpha, \beta, u)}=F \quad \forall\, (\alpha, \beta, u) 
\in N_\Z, \\
(2)
&\; F(h+ac)=e^{ka}F(h) \quad \forall\, h \in Y \; \text{and} \; a \in \C. 
\end{matrix}
\right\},  \right.
\]
and set
\[ \tTh:=\bigoplus_{k=0}^\infty \tTh_k. \]
An element $F$ of $\tTh_k$ is called a theta function of degree $k$.
\end{Def}
\begin{rem} 
The ring $\tTh$ is a $\Z$-graded algebra over the ring $\tTh_0=\cO_{\nh}$, 
the ring of holomorphic functions on the upper half plane $\nh$.
\end{rem}
A typical example of an element of $\tTh_k$ for $k \in \Z_{>0}$ is the classical theta 
function $\Theta_\lambda$ of degree $k$, i.e., $\lambda \in P^k$. 

\begin{rem}
It follows from \eqref{def_theta-function} that
\[ 
\Theta_{\lambda+k\alpha+a\delta}=\Theta_{\lambda}\quad 
\mbox{for every $\alpha\in M$ and $a\in\C$}.
\]
Thus,  a classical theta function of degree $k$ depends only on 
the finite set $P^k\,\mbox{mod}\,(kM+\C\delta)$.  
\end{rem}
Let $D$ be the Laplacian on $Y$: 
\[ D=\frac{1}{4\pi^2}\left(2\frac{\partial}{\partial t}\frac{\partial}{\partial \tau}
-\sum_{i=1}^l \frac{\partial^2}{\partial z_i^2}\right). 
\]
For $k \in \Z_{\geq 0}$, set
\[  Th_k:=\begin{cases}  \; \{F \in \tTh_k\, \vert D(F)=0\, \} \qquad & k>0, \\
                                    \;  \C \qquad & k=0. \end{cases}
\]
V. Kac and D. Peterson showed the following proposition.
\begin{prop}[\cite{KP}]\label{prop_base-theta} Let $k$ 
be a positive integer. 
\begin{enumerate}
\item[(i)] The set $\{ \Theta_\lambda \vert \lambda \in P^k \; \text{mod}\; kM+\C 
\delta \, \}$ is a $\C$-basis of $Th_k$. 
\item[(ii)] The map: $\cO_\nh \otimes_{\C} Th_k \rightarrow \tTh_k$ defined by 
$f \otimes F \mapsto fF$ is an isomorphism of $\cO_\nh$-modules. In other words, 
$\tTh_k$ is a free $\cO_{\nh}$-module with basis 
$\{\Theta_\lambda\,  \vert \, 
\lambda \in P^k \, \mathrm{mod}\,  (kM +\C \delta) \, \}$.
\end{enumerate}
\end{prop}

As the finite Weyl group $W_f$ acts on $Y$, it induces the right action of $W_f$ 
on $\cO_Y$: 
\[F\vert_w(h):=F(w\cdot h)\quad \mbox{for $F \in \cO_Y$ and $w \in W_f$. }\]

For any $k \in \Z_{\geq 0}$, this right $W_f$-action on $\cO_Y$ restricts to a right 
$W_f$-action on $\tTh_k$. Thus, we set
\begin{align*}
\tTh_k^+:=& \{ F \in \tTh_k \, \vert \, F\vert_w=F \quad \forall \, w \in W_f\}, \\
\tTh_k^-:=
& \{ F \in \tTh_k \, \vert \, F\vert_w=\det{}_{\fh_f^\ast}(w)F 
\quad \forall \, w \in W_f \}, \\
Th_k^\pm:=&\tTh_k^\pm \cap Th_k.
\end{align*}
An element of $\tTh^\pm:=\bigoplus_{k \in \Z_{\geq 0}} \tTh_k^\pm$ is called a 
$W_f$-invariant (resp. $W_f$-anti-invariant). For $k \in \Z_{>0}$, set
\[ P_{++}^k:=\{ \Lambda \in P^k\, \vert \, \langle h_i, \Lambda\rangle \in \Z_{>0}\}. \]

Let $\Lambda\in P^k$. Obviously, the element $A_{\Lambda+\rho}$ introduced in
\S \ref{sect_ch} is a $W_f$-anti-invariant. On the other hand, set 
\[S_{\Lambda}
:=e^{-\frac{I(\Lambda,\Lambda)}{2k}\delta}\sum_{w\in W}e^{w(\lambda)}=
\sum_{w_f\in W_f}\Theta_{w_f(\Lambda)}
.\] 
Then, it is an element of $Th_k^+$. 

Another important example of $W_f$-invariants is the normalized character 
$\chi_{\Lambda}$. Indeed, Proposition \ref{prop:ch} (ii)
tells us that $\chi_{\Lambda}$ is an element of $\tTh_k$. In addition, thanks to
the description \eqref{Weyl-Kac-A} of $\chi_{\Lambda}$, it is invariant under the 
action of $W_f$. Therefore, one has $\chi_{\Lambda}\in \tTh_k^+$.
\begin{prop}[\cite{KP}] 
Let $k$ be a non-negative integer.
\begin{enumerate}
\item[(i)] The set $\{ S_{\Lambda} \vert \, \Lambda \in P_+^k \, 
\text{mod}\,  \C \delta\,  \}$
is a $\C$-basis of $Th_k^+$. 
\item[(ii)] For $k\geq 2l+1$, the set $\{A_\lambda \vert \, \lambda \in P_{++}^k\, 
\text{mod}\, \C \delta\, \}$ is a $\C$-basis of $Th_k^-$.
\item[(iii)] The map $\Phi_k$ in Proposition \ref{prop_base-theta} is 
$W_f$-equivariant. Therefore, the $\cO_\nh$-modules $\tTh_k^\pm$ are free 
over $\{ \chi_\Lambda \vert \, \Lambda \in P_+^k \, \text{mod}\,  \C \delta\,  \}$ 
$($resp. $\{A_\lambda \vert \, \lambda \in P_{++}^k\, \text{mod}\, \C \delta\, \}$$)$. 
\end{enumerate}
\end{prop}
As the map $P_+^k \rightarrow P_{++}^{k+2l+1}; \lambda \mapsto \lambda+\rho$ 
is bijective,  \eqref{Weyl-Kac-A} and the above proposition implies
\begin{cor}
The space of $W_f$-anti-invariants $\tTh^-$ is a free $\tTh^+$-module over 
$A_\rho$. 
\end{cor}
As $P_+ \, \text{mod}\,  \C \delta$ is the monoid generated by 
$\{\Lambda_i\}_{0\leq i\leq l}$ and $\langle c, \Lambda_i \rangle=a_i^\vee$, it can 
be shown that the Poincar\'{e} series $P_{\tTh^+}(T)$ of the graded 
$\cO_\nh$-algebra $\tTh^+$ is given by 
\[ P_{\tTh^+}(T):=\sum_{k \in \Z_{\geq 0}} \rank_{\cO_\nh}(\tTh^+_k)T^k
=\prod_{i=0}^l(1-T^{a_i^\vee})^{-1}. \] 
By the same computation, this implies
\begin{equation}\label{Poincare-ser-Th^+}
 \sum_{k \in \Z_{\geq 0}} \dim_\C(Th_k^+)T^k=\prod_{i=0}^l(1-T^{a_i^\vee})^{-1}. 
\end{equation}

For each $\tau \in \nh$, let $Y_\tau$ be the subset of $Y$ that 
corresponds to $\{ \tau\} \times \fh_f \times \C$ via the isomorphism 
$Y \cong \nh \times \fh_f \times \C$ discussed in \eqref{coord-system-II}
, $\iota_{Y_\tau}:Y_{\tau}\hookrightarrow Y$ the corresponding 
embedding. The above computation leads to the next conjecture: 
\begin{conj}[cf. \cite{KP}] \label{conj-KP}
Let $\fg$ be the affine Lie algebra of type $A_{2l}^{(2)}$. 
\begin{enumerate} 
\item[(i)] For each $\tau \in \nh$, $\tTh^+\vert_{Y_\tau}$ is a 
polynomial ring generated by $\{
\iota_{Y_\tau}^*(\chi_{\Lambda_i})\}_{0\leq i\leq l}$
over $\C$.
\item[(ii)] The graded ring $\tTh^+$ is a polynomial algebra
over $\cO_\nh$ generated by $\{\chi_{\Lambda_i}\}_{0\leq i\leq l}$.
\end{enumerate}
\end{conj}

%----------------------------------------------
\subsection{Jacobian of fundamental characters}
Let $\Lambda \in P_+$ be a dominant integral weight. For an integer 
$0\leq i\leq l$, we define the directional derivative $\partial_i \chi_\Lambda$ by
\[ (\partial_i \chi_\Lambda)(\tau, z, t):=\lim_{s \rightarrow 0} 
\frac{\chi_\Lambda(h+2\pi\sqrt{-1}sh_j)-\chi_\Lambda(h)}{s}. \]
As $c=h_0+2\sum_{i=1}^l h_i$, it follows that
\[ (\partial_0 \chi_{\Lambda})(\tau, z, t)=2\pi\sqrt{-1}\langle c, \Lambda \rangle 
\chi_\Lambda (\tau, z,t) -2\sum_{i=1}^l\partial_i\chi_{\Lambda}(\tau, z, t). \]
Hence, the Jacobian of the fundamental characters is, by definition, 
\[ J(\tau, z, t):=\det (\partial_j \chi_{\Lambda_i})_{i,j=0,1,\cdots, l}=
2\pi\sqrt{-1}
\begin{vmatrix}  \chi_0 & \partial_1\chi_0 & \cdots & \partial_l \chi_0 \\
                         2 \chi_1 & \partial_1\chi_1 & \cdots & \partial_l \chi_1 \\
                         \vdots & \vdots & & \vdots \\
                         2 \chi_l  & \partial_1\chi_l & \cdots & \partial_l \chi_l
                          \end{vmatrix}.
                          \]
It can be easily seen that this determinant is an element of $\tTh_{2l+1}^-$, i.e., 
there exists a holomorphic function $F \in \cO_\nh$ such that  
$J(\tau, z, t)=F(\tau)A_\rho(\tau, z, t)$. 

In the rest of this article, we will determine this function $F$ and prove 
Conjecture \ref{conj-KP}.

\begin{rem}
A weaker statement is proved by I. Bernstein and O. Schwarzmann 
\cite{BS1} and \cite{BS2} for any affine root systems except for $D_l^{(1)}$ and 
$A_{2l}^{(2)}$. 
\end{rem}

\section{Preliminary computations}
In this section, we study the modular transformations and leading terms of the Jacobian of certains characters. 
\subsection{Explicit formulas on $a(\lambda, \mu)$}
As $a_0^\vee=1$ and $a_i^\vee=2$ for $0<i\leq l$, one sees that 
\begin{enumerate}
\item[(i)] $P_+^1 \mod \C \delta =\{\Lambda_0\}$ and 
\item[(ii)] $P_+^2 \mod \C \delta = \{2 \Lambda_0\} \cup 
\{ \Lambda_i\}_{1\leq i\leq l}$. 
\end{enumerate}
\vskip 3mm

First, we compute $a(\Lambda_0, \Lambda_0)$. 
By Theorem \ref{thm_KP} and the 
denominator identity, we have
\begin{align*}
a(\Lambda_0, \Lambda_0) 
=
&(\sqrt{-1})^{l^2}(2(l+1))^{-\frac{1}{2}l}\sum_{w \in W_f} \vep(w)
\exp \left(-\frac{\pi\sqrt{-1}I(w(\overline{\rho}), \overline{\rho})}{l+1}\right) \\
=
&(\sqrt{-1})^{l^2}(2(l+1))^{-\frac{1}{2}l}
\prod_{\alpha \in \Delta_f^+}(e^{\frac{1}{2}\alpha}-e^{-\frac{1}{2}\alpha})\left(-
\frac{\pi\sqrt{-1}}{l+1}\overline{\rho}\right) \\
=
&2^{l^2-\frac{1}{2}l}(l+1)^{-\frac{1}{2}l}
\prod_{\alpha \in \Delta_f^+}\sin \left(\frac{I(\overline{\rho},\alpha)}{2(l+1)}\pi\right). 
\end{align*}
Hence, by \S \ref{sect_affine-Weyl} and the well-known formula
\begin{equation}\label{sine-gamma}
\prod_{k=1}^{n-1} \sin \left( \frac{k}{n}\pi\right)=\frac{n}{2^{n-1}}
\end{equation}
for $n \in \Z_{>1}$, it follows that
\begin{align*}
\prod_{\alpha \in \Delta_f^+}\sin \left(\frac{I(\overline{\rho},\alpha)}{2(l+1)}\pi\right) 
=
&\prod_{1\leq i<j\leq l}\sin\left(\frac{j-i}{2(l+1)}\pi\right)\sin\left(\frac{j+i}{2(l+1)}
\pi\right)
\prod_{i=1}^l \sin\left(\frac{i}{l+1}\pi\right) \\
=
&2^{-l^2+\frac{1}{2}l}\cdot (l+1)^{\frac{1}{2}l},
\end{align*}
thus we obtain 
\begin{equation}\label{S-level-1}
a(\Lambda_0, \Lambda_0)=1. 
\end{equation}

Next, for level $2$ case, we set 
\[ \lambda_i=\begin{cases} 2\Lambda_0 \qquad & i=0, \\ \Lambda_i \qquad & 
0<i\leq l.
\end{cases}
\]
By Remark \ref{rem_fin-Weyl}, for $\overline{\lambda+\rho}=\sum_{i=1}^l m_i 
\vep_i$ and $\overline{\mu+\rho}=\sum_{i=1}^l n_i \vep_i$, we have
\begin{align*}
&\sum_{w \in \fS_l \ltimes (\Z/2\Z)^l}\vep(w)\exp\left(-\frac{2\pi\sqrt{-1}
I(\overline{\lambda+\rho}, w(\overline{\mu+\rho}))}{2l+3}\right) \\
&\qquad =
\sum_{\sigma \in \fS_l} \vep(\sigma)
\prod_{r=1}^l \left(\exp\left(-\frac{2\pi\sqrt{-1}}{2l+3}m_in_{\sigma^{-1}(r)}\right)-
\exp\left(\frac{2\pi\sqrt{-1}}{2l+3}m_rn_{\sigma^{-1}(r)}\right)\right)\\
&\qquad =(-2\sqrt{-1})^l \sum_{\sigma \in \fS_l}\vep(\sigma)
\prod_{r=1}^l \left(2\sin\left(\frac{2m_rn_{\sigma^{-1}(r)}}{2l+3}\pi \right)\right)\\
&\qquad
=(-2\sqrt{-1})^l \det\left(\sin\left( \frac{2m_rn_s}{2l+3}\pi\right)
\right)_{1\leq r,s\leq l}.
\end{align*}
As $\overline{\Lambda_i+\rho}=\sum_{k\leq i}(l+2-k)\vep_k+\sum_{k>i}(l+1-k)
\vep_k$ and
\[ \det\left(\sin\left( \frac{2m_rn_s}{2l+3}\pi\right)\right)_{1\leq r,s\leq  l}=
\det\left(\sin\left( \frac{2m_{l+1-r}n_{l+1-s}}{2l+3}\pi\right)\right)_{1\leq r,s\leq  l}, 
\]
to compute $a(\lambda_i, \lambda_j)$, it suffices to compute the $(-1)^{i+j}$ times 
$(l+1-i,l+1-j)$-entry of the cofactor matrix of the matrix
\[ M_2:=\left( \sin\left(\frac{2rs}{2l+3}\pi\right)
\right)_{1\leq r,s \leq  l+1}. \] 
With the aid of $C_l$-type denominator identity (cf. \cite{Kr}), i.e., the identity 
\[\det(X_i^j-X_i^{-j})_{1\leq i,j\leq l}=
(\prod_{i=1}^l X_i)^{-l}\prod_{1\leq i<j\leq l}(1-X_iX_j)(X_i-X_j)\prod_{i=1}^l(X_i^2-1).
\]
in $\C[X_1^{\pm 1}, \cdots, X_l^{\pm 1}]$, one can show that
\[ \det M_2
=(-1)^{\frac{1}{2}l(l+1)}2^{-(l+1)}(2l+3)^{\frac{1}{2}(l+1)}. \]
Moreover, it can be shown that 
\[ M_2^2=\frac{1}{4}(2l+3)I_{l+1}. \]
Combining these facts, we obtain
\begin{equation}\label{S-level-2}
 a(\lambda_i,\lambda_j)=\frac{2}{\sqrt{2l+3}}\cos\left(\frac{(2i+1)(2j+1)}{2(2l+3)}
 \pi\right) \qquad (0\leq i,j\leq l). 
 \end{equation}
 %---------------------------------------------------------
\subsection{Modular transformation of Jacobians}
For $0\leq i\leq l$, set
\[J^{i}(\tau, z, t):=2\pi \sqrt{-1}\begin{vmatrix}
\chi_{\Lambda_0} & \partial_1 \chi_{\Lambda_0} & \cdots & 
\partial_l\chi_{\Lambda_0} \\
2\chi_{\lambda_{i_1}} & \partial_1 \chi_{\lambda_{i_1}} & \cdots & \partial_l 
\chi_{\lambda_{i_1}} \\
\vdots & \vdots & & \vdots \\
2\chi_{\lambda_{i_l}} & \partial_1 \chi_{\lambda_{i_l}} & \cdots & \partial_l 
\chi_{\lambda_{i_l}} \\
\end{vmatrix} \qquad (\tau, z, t) \in \nh \times \fh_f \times \C, \]
where $i_1, i_2, \cdots i_l$ are integers such that 
$i_1<i_2<\cdots <i_l$ and 
$\{i_1, i_2, \cdots, i_l\}=\{0, 1, \cdots, l\} \setminus \{i \}$. Note that 
\[J^0(\tau, z, t)=J(\tau, z, t).\]

Since these determinants are $W$-anti invariant and their degrees are 
$\sum_{i=0}^l a_i^\vee=2l+1$, it follows that $J^{i}(\tau, z, t) \in \tTh^-_{2l+1}
=\cO_\nh A_\rho$, i.e., there exist holomorphic functions $F^{i} \in \cO_\nh$ 
such that 
\begin{equation}\label{Jac1}
J^{i}(\tau, z, t)=F^{i}(\tau)A_\rho(\tau, z, t), 
\end{equation}
Below, we determine these holomorphic functions $\{ F^{i}(\tau)\}_{0\leq i\leq l}$ 
explicitly. For this purpose, we compute the modular transformations of 
$J^{i}(\tau, z, t)$. \\

Let $\lambda \in P_+$. Recall that, by Theorem \ref{thm_KP}, there exists a 
unitary matrix $(a(\lambda, \mu))_{\lambda, \mu \in P_+^{\langle c, \lambda\rangle} 
\mod \C \delta}$
such that
\[ \chi_\lambda\left( -\frac{1}{\tau}, \frac{z}{\tau}, t-\frac{
I^{\ast} (z,z)}{2\tau}\right)
=\sum_{\mu \in P_+^{\langle c, \lambda\rangle}} a(\lambda, \mu)
\chi_\mu(\tau, z,t). \]
Differentiating both sides of this formula, we obtain
\begin{align*} (\partial_i \chi_\lambda)\left(-\frac{1}{\tau}, \frac{z}{\tau}, 
t-\frac{I^{\ast}(z,z)}{2\tau}\right) =
&
\tau \sum_{\mu \in P_+^{\langle c, \lambda\rangle}}a(\lambda, \mu)
(\partial_i\chi_\mu)(\tau,z,t) \\
&
+2\pi\sqrt{-1}I^\ast(z,\alpha_i^\vee)\langle c, \lambda \rangle \sum_{\mu \in 
P_+^{\langle c, \lambda\rangle}} a(\lambda,\mu)\chi_\mu(\tau,z,t),
\end{align*}
for any $0< i \leq l$. \\

Set $M_S:=(a(\lambda_i, \lambda_j))_{0\leq i,j\leq l}$. 
Let $ \tilde{M}_S=\bigl((\tilde{M}_S)_{i,j}\bigr)_{0\leq i,j\leq l}$ be 
the cofactor matrix of $M_S$. 
For $0\leq i\leq l$, let $0\leq i_1< i_2<\cdots < i_l$ be non-negative integers such 
that $\{i_1, i_2, \cdots, i_l\}=\{0,1,\cdots, l\} \setminus \{ i\}$. 
By direct calculation, one obtains
\begin{align*}
&J^{i}\left(-\frac{1}{\tau}, \frac{z}{\tau}, t-\frac{I^\ast(z,z)}{2\tau}\right) \\
&=
2\pi \sqrt{-1}a(\Lambda_0, \Lambda_0)\tau^l\begin{vmatrix}
\chi_{\Lambda_0} & \partial_1\chi_{\Lambda_0} & \cdots & \partial_l 
\chi_{\Lambda_0} \\
2\sum_{j}a(\lambda_{i_1}, \lambda_j)\chi_{\lambda_j}
&\sum_{j}a(\lambda_{i_1}, \lambda_j)\partial_1\chi_{\lambda_j} & \cdots 
& \sum_{j}a(\lambda_{i_1}, \lambda_j)\partial_l\chi_{\lambda_j} \\
\vdots & \vdots & & \vdots \\
2\sum_{j}a(\lambda_{i_l}, \lambda_j)\chi_{\lambda_j}
&\sum_{j}a(\lambda_{i_l}, \lambda_j)\partial_1\chi_{\lambda_j} & \cdots 
& \sum_{j}a(\lambda_{i_l}, \lambda_j)\partial_l\chi_{\lambda_j}
\end{vmatrix} \\
&=\tau^l \sum_{j=0}^l (-1)^{i+j}(\tilde{M}_S)_{i,j}J^{j}(\tau, z, t).
\end{align*}

On the other hand, it can be checked that 
$M_S^2=I_{l+1}$. 
Recall the super denominator identity of type $B(0,n)$, i.e., the next 
identity in $\C[X_1^{\pm \frac{1}{2}}, \cdots X_n^{\pm \frac{1}{2}}]$:
\begin{equation}\label{denom-B(0,n)} \det(X_i^{j-\frac{1}{2}}+X_i^{-j+\frac{1}{2}})_{1\leq i,j\leq n}=(\prod_{i=1}^n X_i)^{-n+
\frac{1}{2}}
\prod_{1\leq i<j\leq n}(1-X_iX_j)(X_i-X_j)\prod_{i=1}^n(1+X_i).
\end{equation}
With the aid of this formula, we can show
\begin{lemma} 
$\det M_S=(-1)^{\binom{l+1}{2}}$. 
\end{lemma}
\begin{proof}
Setting $\omega=\exp\left(\frac{2\pi\sqrt{-1}}{2l+3}\right)$, $n=l+1$ and $X_i=\omega^{i-\frac{1}{2}}$ ($1\leq i\leq l+1$) in \eqref{denom-B(0,n)}, we have
\begin{align*}
\det(M_s)=
&(2l+3)^{-\frac{1}{2}(l+1)}
\det\left(\omega^{(i-\frac{1}{2})(j-\frac{1}{2})}+\omega^{-(i-\frac{1}{2})(j-\frac{1}{2})}\right)_{1\leq i,j\leq l+1} \\
=
&(2l+3)^{-\frac{1}{2}(l+1)}\left(\prod_{i=1}^{l+1}\omega^{i-\frac{1}{2}}\right)^{-(l+\frac{1}{2})}
\prod_{1\leq <i<j\leq l+1}(1-\omega^{i+j-1})(\omega^{i-\frac{1}{2}}-\omega^{j-\frac{1}{2}})\prod_{i=1}^{l+1}(1+\omega^{i-\frac{1}{2}}).
\end{align*}
Now, by direct computation, it follows that
\begin{align*}
&\left(\prod_{i=1}^{l+1}\omega^{i-\frac{1}{2}}\right)^{-(l+\frac{1}{2})}
\prod_{1\leq <i<j\leq l+1}(1-\omega^{i+j-1})(\omega^{i-\frac{1}{2}}-\omega^{j-\frac{1}{2}})\prod_{i=1}^{l+1}(1+\omega^{i-\frac{1}{2}}) \\
=
&(-1)^{\binom{l+1}{2}}\cdot 2^{(l+1)^2} 
\prod_{1\leq i<j\leq l+1}\sin\left(\frac{i+j-1}{2l+3}\pi\right)
\sin\left(\frac{j-i}{2l+3}\pi\right) 
\prod_{i=1}^{l+1}\cos\left(\frac{2i-1}{2l+3}\cdot \frac{\pi}{2} \right) \\
=
&(-1)^{\binom{l+1}{2}}\cdot 2^{(l+1)^2} \left(\prod_{i=1}^{l+1}\sin\left(\frac{i}{2l+3}\pi\right)\right)^{l+1}.
\end{align*}
By \eqref{sine-gamma},  we obtain
\[
\det(M_S)=
(2l+3)^{-\frac{1}{2}(l+1)}\cdot (-1)^{\binom{l+1}{2}}\cdot 2^{(l+1)^2} 
\left(\prod_{i=1}^{2l+2}\sin\left(\frac{i}{2l+3}\pi\right)\right)^{\frac{1}{2}(l+1)} \\
=
 (-1)^{\binom{l+1}{2}}.
\]
\end{proof}

Hence, we obtain
\begin{align*}
(\tilde{M}_S)_{i,j}&=
(-1)^{\binom{l+1}{2}}(M_S)_{i,j}\\
&=(-1)^{\binom{l+1}{2}}a(\lambda_i, \lambda_j) \\
&=(-1)^{\binom{l+1}{2}}\frac{2}{\sqrt{2l+3}}\cos\left(\frac{(2i+1)(2j+1)}{2(2l+3)}
\pi\right). 
\end{align*}
Thus, we obtain the next formula:
\begin{equation}\label{Jac-S}
 \begin{split}
 &(-1)^{i}J^{i}\vert_S(\tau, z, t)\\
 &=(-1)^{i}J^{i}\left(-\frac{1}{\tau}, \frac{z}{\tau}, t-
 \frac{I^\ast(z,z)}{2\tau}\right) \\
 &=
 \tau^l (-1)^{\binom{l+1}{2}}\frac{2}{\sqrt{2l+3}}\sum_{j=0}^l
 \cos\left(\frac{(2i+1)(2j+1)}{2(2l+3)}\pi\right)(-1)^jJ^{j}(\tau, z, t). 
 \end{split}
\end{equation}

By definition, the conformal anomaly of the weights 
$\Lambda_0, \lambda_i (0\leq i \leq l)$ are given by
\[ m_{\Lambda_0}=-\frac{1}{24}l, \qquad m_{\lambda_i}
=\frac{1}{2(2l+3)}(2(l+1)i-i^2)-\frac{l(l+1)}{6(2l+3)}, \]
which implies 
\[ m_{\Lambda_0}+\sum_{i=0}^l m_{\lambda_i}=\dfrac{1}{24}l(2l+1). 
\]
Thus, we obtain the next formula:
\begin{equation}\label{Jac-T}
\begin{split}
(-1)^{i}J^{i}\vert_{T^2}(\tau, z, t)&=(-1)^{i}J^{i}( \tau+2, z, t)\\
&=\exp\left(4\pi\sqrt{-1}\left(\frac{1}{24}l(2l+1)-m_{\lambda_i}\right)\right)
(-1)^{i}J^{i}( \tau, z, t).
\end{split}
\end{equation}
%-----------------------------------
\subsection{Leading terms of Jacobians}\label{sect:lt-Jacobian}
In this subsection, we compute the leading degree of each Jacobian $
J^{i}(\tau, z, t)$ ($0\leq i\leq l$) as a $q$-series, where $q:=
e^{\pi\sqrt{-1}\tau}$. 
Let
\begin{align*}
& P^f=\{ \, \lambda \in \fh_f^\ast\,  \vert\,  \langle h_i, \lambda \rangle \in \Z 
\quad 1\leq \forall\, i \leq l\, \} \\
& P_+^f=\{ \, \lambda \in \fh_f^\ast\,  \vert\,  \langle h_i, \lambda \rangle \in 
\Z_{\geq 0} \quad 1\leq \forall\, i \leq l\, \}
\end{align*}
be the set of integral (resp. dominant integral) weights of $\fh_f$. The Weyl group 
acts on $P^f$, hence on its group algebra 
\[ \C[P^f]=\left\{\, \sum_{\lambda \in P^f} c_\lambda e^\lambda\, \left\vert \, 
\begin{matrix} \text{i)} & c_\lambda \in \C,  \\ \text{ii)} & \sharp \{ \lambda \,\vert 
\,c_\lambda \neq 0\}<\infty.
\end{matrix} \right\} \right.
\]
For $\lambda \in P^f$, set $a_\lambda:=\sum_{w \in W_f} \vep(w)e^{w(\lambda)}$. 
The set $\{ a_{\lambda+\overline{\rho}} \}_{\lambda \in P_+^f}$ spans the set of 
set of $W_f$-skew invariants $\C[P^f]^{- W_f}$ and it satisfies 
$a_{w(\lambda)}=\vep(w)a_{\lambda}$. For $\lambda \in P^f$, set 
$\chi_\lambda^f:=\dfrac{a_{\lambda+\overline{\rho}}}{a_{\overline{\rho}}}$. It follows 
that $\chi_\lambda^f\in \C[P^f]^{W_f}$ is the Weyl character formula for the finite 
root system $\Delta_f$ when $\lambda \in P_+^f$. \\

Now, regard $e^\lambda$ ($\lambda \in P^f$) as functions on $\fh_f$ defined by $h 
\mapsto e^{\langle h, \lambda \rangle}$. We rewrite the anti-invariant 
$A_{\Lambda+\rho}$  ($\Lambda \in P_k^+$) of the affine Weyl group $W$:
\begin{align*}
&A_{\Lambda+\rho}(\tau, z, t) \\
&=e^{2(k+2l+1)\pi\sqrt{-1}t}
\sum_{w \in W_f}\vep(w)\sum_{\alpha \in M}
q^{\frac{1}{2(k+2l+1)}\left\vert (k+2l+1)\alpha+w(\overline{\Lambda}+
\overline{\rho})\right\vert^2}e^{2\pi\sqrt{-1}\langle z, 
(k+2l+1)\alpha+w(\overline{\Lambda}+\overline{\rho}) \rangle} \\
&=e^{2(k+2l+1)\pi\sqrt{-1}t}
\sum_{\alpha \in M}
q^{\frac{1}{2(k+2l+1)}\left\vert (k+2l+1)\alpha+\overline{\Lambda}+\overline{\rho}
\right\vert^2}\sum_{w \in W_f}\vep(w)e^{2\pi\sqrt{-1}\langle z, w((k+2l+1)\alpha+
\overline{\Lambda}+\overline{\rho}) \rangle} \\
&= e^{2(k+2l+1)\pi\sqrt{-1}t}
\sum_{\gamma \in M+\frac{\overline{\Lambda}+\overline{\rho}}{k+2l+1}}
q^{\frac{1}{2}(k+2l+1)I(\gamma, \gamma)}\sum_{w \in W_f}
\vep(w)e^{2\pi\sqrt{-1}\langle z, w((k+2l+1)\gamma) \rangle} \\
&=e^{2(k+2l+1)\pi\sqrt{-1}t}
\sum_{\gamma \in M+\frac{\overline{\Lambda}+\overline{\rho}}{k+2l+1}}
a_{(k+2l+1)\gamma}(2\pi\sqrt{-1}z)q^{\frac{1}{2}(k+2l+1)I(\gamma, \gamma)},
\end{align*}
where $\vert \beta \vert^2:=I(\beta, \beta)$ for $\beta \in \fh_f^\ast$. 
In particular, for $\Lambda=0, \Lambda_0, \lambda_i\; 
(0\leq  i\leq l)$, the first few terms of this 
formula are given as follows.

\[\begin{aligned}
&e^{-2(2l+1)\pi\sqrt{-1}t}q^{-\frac{1}{12}l(l+1)} A_\rho(\tau,z,t) \\
&\quad =
a_{\overline{\rho}}(2\pi\sqrt{-1}z)
-a_{\overline{\Lambda}_1+\overline{\rho}}(2\pi\sqrt{-1}z)q^{\frac{1}{2}}
+a_{\overline{\Lambda}_1+\overline{\Lambda}_2+\overline{\rho}}
(2\pi\sqrt{-1}z)q^{\frac{3}{2}}\\
&\hspace*{70mm} -a_{2\overline{\Lambda_2}+\overline{\rho}}
(2\pi\sqrt{-1}z)q^2+O(q^{\frac{5}{2}}),  
\end{aligned}\leqno{\text{ }(\Lambda=0):}\]
\[\begin{aligned}
&e^{-4(l+1)\pi\sqrt{-1}t}q^{-\frac{1}{24}l(2l+1)}A_{\Lambda_0+\rho}(\tau,z,t) \\
&\quad
=a_{\overline{\rho}}(2\pi\sqrt{-1}z)-a_{2\overline{\Lambda}_1+\overline{\rho}}
(2\pi\sqrt{-1}z)q
+a_{2\overline{\Lambda_1}+\overline{\Lambda_2}+\overline{\rho}}
(2\pi\sqrt{-1}z)q^2+O(q^3), 
\end{aligned}\leqno{\text{ }(\Lambda=\Lambda_0}):\]
\[\begin{aligned}
&e^{-2(2l+3)\pi\sqrt{-1}t}q^{-\frac{l(l+1)(2l+1)}{12(2l+3)}}A_{2\Lambda_0+\rho}
(\tau,z,t) \\
&\quad =
a_{\overline{\rho}}(2\pi\sqrt{-1}z)-a_{3\overline{\Lambda}_1+\overline{\rho}}
(2\pi\sqrt{-1}z)q^{\frac{3}{2}}+a_{3\overline{\Lambda_1}+\overline{\Lambda_2}+
\overline{\rho}}(2\pi\sqrt{-1}z)q^{\frac{5}{2}}+O(q^{\frac{7}{2}}), 
\end{aligned}\leqno{\text{ }(\Lambda=2\Lambda_0):}\]
\[\begin{aligned}
&e^{-2(2l+3)\pi\sqrt{-1}t}q^{-\frac{l(l+1)(2l+1)}{12(2l+3)}-\frac{i(2(l+1)-i)}{2(2l+3)}}
A_{\Lambda_i+\rho}(\tau,z,t) \\
&\quad = a_{\overline{\Lambda_i}+\overline{\rho}}(2\pi\sqrt{-1}z)-
a_{\overline{\Lambda}_1+
\overline{\Lambda}_i+\overline{\rho}}(2\pi\sqrt{-1}z)q^{\frac{1}{2}}\\
&\hspace*{30mm}+
\begin{cases} 
a_{2(\overline{\Lambda_1}+\overline{\Lambda_2})+\overline{\rho}}
(2\pi\sqrt{-1}z)q^{\frac{5}{2}}+ O(q^3) \quad &i=1 \\
a_{2\overline{\Lambda_1}+\overline{\Lambda_2}+\overline{\Lambda_i}+
\overline{\rho}}(2\pi\sqrt{-1}z)q^{\frac{3}{2}}+O(q^2) \quad &i>1. \end{cases}\\
\end{aligned}\leqno{\text{ }(\Lambda=\lambda_i,i\ne 0):}\]

Thus, the normalized characters have the next expansions: 
\begin{align*}
&e^{-4\pi\sqrt{-1}t}q^{\frac{l(l+1)}{6(2l+3)}}\chi_{2\Lambda_0}(\tau,z,t) \\
&\qquad =
1+\chi_{\overline{\Lambda_1}}^f(2\pi\sqrt{-1}z)q^{\frac{1}{2}}+
(\chi_{\overline{\Lambda_1}}^f)^2(2\pi\sqrt{-1}z)q \\
&\hspace*{50mm}
+\left( (\chi_{\overline{\Lambda_1}}^f)^3-\chi_{3\overline{\Lambda_1}}^f-
\chi_{\overline{\Lambda_1}+\overline{\Lambda_2}}^f
\right)(2\pi\sqrt{-1}z)q^{\frac{3}{2}}+O(q^{2})
\end{align*}
and, for $0< i\leq l$,  
\begin{align*}
&e^{-4\pi\sqrt{-1}t}q^{\frac{3i(i-2(l+1))+l(l+1)}{6(2l+3)}}\chi_{\lambda_i}(\tau,z,t) \\
&\qquad =\chi_{\overline{\Lambda_i}}^f(2\pi\sqrt{-1}z)
+(\chi_{\overline{\Lambda_1}}^f\chi_{\overline{\Lambda_i}}^f-
\chi_{\overline{\Lambda_1}+\overline{\Lambda_i}}^f)(2\pi\sqrt{-1}z)q^{\frac{1}{2}}\\
&\hspace*{50mm}
+\left(\chi_{\overline{\Lambda_1}}^f
(\chi_{\overline{\Lambda_1}}^f\chi_{\overline{\Lambda_i}}^f
-\chi_{\overline{\Lambda_1}+\overline{\Lambda_i}}^f)\right)(2\pi\sqrt{-1}z)q
+O(q^{\frac{3}{2}}).
\end{align*}

Notice that $\chi_{\overline{\Lambda_1}}^f \chi_{\overline{\Lambda_i}}^f
-\chi_{\overline{\Lambda_1}+\overline{\Lambda_i}}^f
=\chi_{\overline{\Lambda_{i-1}}}^f +\chi_{\overline{\Lambda_{i+1}}}^f$ 
where we set $\chi_{\overline{\Lambda_{l+1}}}^f=0$. 
\vskip 0.1in

Now, we analyze the first few terms of $\chi_{\Lambda_0}$. 
Recall that the character of the fundamental representation $L(\Lambda_0)$ has
the special expression (cf. \cite{KP}):
\[ \ch L(\Lambda_0)=\frac{\sum_{\alpha \in M}e^{t_\alpha(\Lambda_0)}}
{\prod_{r=1}^\infty(1-e^{-r\delta})^l}. \]
Hence, it follows from the Jacobi triple product identity that
\[ \chi_{\Lambda_0}(\tau,z,t)=e^{2\pi\sqrt{-1}t}q^{-\frac{1}{24}l}
\prod_{i=1}^l\left(\prod_{r>0}(1+q^{r-\frac{1}{2}}e^{2\pi\sqrt{-1}\langle z, \vep_i 
\rangle})
(1+q^{r-\frac{1}{2}}e^{-2\pi\sqrt{-1}\langle z, \vep_i \rangle})\right). 
\]
With this expression, one can derive the leading terms of the normalized character 
$\chi_{\Lambda_0}(\tau,z,t)$ as follows.
 For $r \in \Z_{>0}$, thanks to the $W_f$-invariance, it can be shown that 
\begin{align*}
&\prod_{i=1}^l(1+q^{r-\frac{1}{2}}e^{2\pi\sqrt{-1}\langle z, \vep_i \rangle})
(1+q^{r-\frac{1}{2}}e^{-2\pi\sqrt{-1}\langle z, \vep_i \rangle}) \\
=
&q^{l(r-\frac{1}{2})}\left(\sum_{j=0}^{[\frac{l}{2}]}
\chi_{\overline{\Lambda_{l-2j}}}^f(2\pi\sqrt{-1}z)+
\sum_{i=1}^l (q^{i(r-\frac{1}{2})}+q^{-i(r-\frac{1}{2})})
\sum_{j=0}^{[\frac{l-i}{2}]}\chi_{\overline{\Lambda_{l-i-2j}}}^f(2\pi\sqrt{-1}z)\right) \\
=
&\sum_{i=0}^{2l} 
\left(\sum_{j=0}^{[\frac{\min\{i, 2l-i\}}{2}]}
\chi_{\overline{\Lambda_{ \min\{i, 2l-i\}-2j}}}^f(2\pi\sqrt{-1}z)\right)
q^{i(r-\frac{1}{2})},
\end{align*}
where we set $\chi_{\overline{\Lambda_0}=1}$ and $[x]$ for 
$x \in \R$ signifies the maximal integer $\leq x$. 
Thus, the leading terms of 
\[ e^{-2(2l+1)\pi\sqrt{-1}t}q^{-(m_{\Lambda_0}-m_{\lambda_i}+\sum_{j=0}^l 
m_{\lambda_j})}J^{i}(\tau, z, t)
\]
is given by the leading terms of 
\[ \begin{vmatrix}
1+O(q^{\frac{1}{2}}) & \partial_1
\chi_{\overline{\Lambda_i}}^f(2\pi\sqrt{-1}z)q^{\frac{i}{2}}
+O(q^{\frac{i+1}{2}}) & \cdots & 
\partial_l\chi_{\overline{\Lambda_i}}^f(2\pi\sqrt{-1}z)q^{\frac{i}{2}}
+O(q^{\frac{i+1}{2}}) \\
2+O(q^{\frac{1}{2}}) & \partial_1\chi_{\overline{\Lambda_1}}^f(2\pi\sqrt{-1}z)q
+O(q^{\frac{3}{2}}) & \cdots & 
\partial_l\chi_{\overline{\Lambda_1}}^f(2\pi\sqrt{-1}z)q
+O(q^{\frac{3}{2}}) \\
2\chi_{\overline{\Lambda_1}}^f(2\pi\sqrt{-1}z)+O(q^{\frac{1}{2}}) &
\partial_1\chi_{\overline{\Lambda_1}}^f(2\pi\sqrt{-1}z)+O(q^{\frac{1}{2}}) 
& \cdots & 
\partial_l \chi_{\overline{\Lambda_1}}^f(2\pi\sqrt{-1}z)+O(q^{\frac{1}{2}}) \\
\vdots & \vdots &  & \vdots \\
2\chi_{\overline{\Lambda_{i-1}}}^f(2\pi\sqrt{-1}z)+O(q^{\frac{1}{2}}) &
\partial_1\chi_{\overline{\Lambda_{i-1}}}^f(2\pi\sqrt{-1}z)+O(q^{\frac{1}{2}}) 
& \cdots & 
\partial_l \chi_{\overline{\Lambda_{i-1}}}^f(2\pi\sqrt{-1}z)+O(q^{\frac{1}{2}}) \\
2\chi_{\overline{\Lambda_{i+1}}}^f(2\pi\sqrt{-1}z)+O(q^{\frac{1}{2}}) &
\partial_1\chi_{\overline{\Lambda_{i+1}}}^f(2\pi\sqrt{-1}z)+O(q^{\frac{1}{2}}) 
& \cdots & 
\partial_l \chi_{\overline{\Lambda_{i+1}}}^f(2\pi\sqrt{-1}z)+O(q^{\frac{1}{2}}) \\
\vdots & \vdots &  & \vdots \\
2\chi_{\overline{\Lambda_l}}^f(2\pi\sqrt{-1}z)+O(q^{\frac{1}{2}}) &
\partial_1\chi_{\overline{\Lambda_l}}^f(2\pi\sqrt{-1}z)+O(q^{\frac{1}{2}}) 
& \cdots & 
\partial_l \chi_{\overline{\Lambda_l}}^f(2\pi\sqrt{-1}z)+O(q^{\frac{1}{2}})
\end{vmatrix}.
\]
In particular, the leading degree with respect to $q$ of this determinant is 
$\frac{1}{2}i$. Therefore, since
\[ m_{\Lambda_0}-m_{\lambda_i}+\sum_{j=0}^l m_{\lambda_j}=
\frac{(2i+1)^2}{8(2l+3)}+\frac{1}{24}(l-1)+\frac{1}{12}l(l+1)-\frac{1}{2}i,
\]
we have
\begin{equation}\label{leading-Jac-i}
\begin{split}
J^{i}(\tau, z, t) \; \propto \;
& e^{2(2l+1)\pi\sqrt{-1}t}q^{\frac{(2i+1)^2}{8(2l+3)}+\frac{1}{24}(l-1)
+\frac{1}{12}l(l+1)} \\
&\times \det(\partial_i \chi_{\overline{\Lambda_j}}^f(2\pi\sqrt{-1}z))_{1\leq i,j\leq l}
(1+O(q^{\frac{1}{2}})).
\end{split}
\end{equation} 
%---------------------------------------------
\subsection{Functional equations on $\{F^{i}(\tau)\}_{0\leq i\leq l}$}
By \eqref{Jac1} and Remark \ref{rem_modular-denom}, the equations
\eqref{Jac-S}, \eqref{Jac-T} and \eqref{leading-Jac-i}  imply the next equations:
\begin{equation}\label{Funct-Eq-F-i}
\begin{split}
(-1)^{i}F^{i}\left( -\frac{1}{\tau}\right)=
& \left(\frac{\tau}{\sqrt{-1}}\right)^{\frac{1}{2}l} \frac{2}{\sqrt{2l+3}}
\sum_{j=0}^l \cos\left(\frac{(2i+1)(2j+1)}{2(2l+3)}\pi\right)(-1)^jF^{j}(\tau), \\
 (-1)^{i}F^{i}(\tau+2, z, t)=
& \exp\left(-4\pi\sqrt{-1}\left(m_{\lambda_i}+\frac{1}{24}l\right)\right)(-1)^{i} 
F^{i}(\tau),
\end{split}
\end{equation}
and also with the computation of the leading term of $A_\rho$ in Subsection \ref{sect:lt-Jacobian}, one obtain
\begin{equation}
\begin{split}
(-1)^{i}F^{i}(\tau)\propto
&\; q^{\frac{(2i+1)^2}{8(2l+3)}+\frac{1}{24}(l-1)}(1+O(q^{\frac{1}{2}})),
\end{split}
\end{equation}
where $\propto$ means the left hand side is proportional to the right hand side up to a non-zero scalar factor. 
%------------------------------------
\section{Determination of Jaocbians}
\subsection{Some modular forms}\label{sect_theta}
For $m \in \Z_{>0}$ and $n \in \Z/2m\Z$, let
\[ \theta_{n,m}(\tau):=\sum_{k \in \Z}q^{m\left(k+\frac{n}{2m}\right)^2}, \]
be the classical theta functions. Their modular transformations are given by
\begin{equation}\label{mod-theta}
\begin{split}
&\theta_{n,m}(\tau+1)=e^{\frac{n^2}{2m}\pi\sqrt{-1}}\theta_{n,m}(\tau), \\
&\theta_{n,m}\left(-\frac{1}{\tau}\right)
=\left( \frac{\tau}{2m\sqrt{-1}}\right)^{\frac{1}{2}}\sum_{n' \in \Z/2m\Z}
e^{-\frac{nn'}{m}\pi\sqrt{-1}}\theta_{n',m}(\tau).
\end{split}
\end{equation}
Here, we recall the Jacobi triple product identity: 
\begin{equation}\label{Jacobi-triple}
\prod_{k=1}^\infty(1-p^k)(1-p^{k-1}w)(1-p^kw^{-1})=\sum_{l \in \Z} (-1)^l 
p^{\binom{l}{2}}w^l.
\end{equation}
The Dedekind eta-function is defined as follows: 
\[ \eta(\tau)=q^{\frac{1}{24}}\prod_{n=1}^\infty(1-q^n)=\sum_{m \in \Z} 
(-1)^mq^{\frac{3}{2}\left(m-\frac{1}{6}\right)^2}.
\]
It is a weight $\frac{1}{2}$ modular form: 
\[ \eta(\tau+1)=e^{\frac{\pi\sqrt{-1}}{12}}\eta(\tau), \qquad 
\eta\left(-\frac{1}{\tau}\right)
=\left(\frac{\tau}{\sqrt{-1}}\right)^{\frac{1}{2}}\eta(\tau). \]

Following \cite{KP}, for $M \in \Z_{>0}$ and $0\leq r<M$, set
\[ F_r^{(M)}(\tau):=q^{\frac{1}{8M}(M-2r)^2}
\prod_{\substack{n \equiv 0 \mod M \\ n>0}}(1-q^n)
\prod_{\substack{n \equiv r \mod M \\ n>0}}(1-q^n)
\prod_{\substack{n \equiv -r \mod M \\ n>0}}(1-q^n). \]
By definition, $F_0^{(M)}(\tau)=\eta(M\tau)^3$, and for $0<r<M$, the Jacobi triple 
product identity \eqref{Jacobi-triple} gives
\[ F_r^{(M)}(\tau)=\theta_{M-2r,2M}(\tau)-\theta_{M+2r,2M}(\tau). \]
In particular, one has
\[ F_r^{(M)}(\tau+1)=e^{\frac{(M-2r)^2}{4M}\pi\sqrt{-1}}F_r^{(M)}(\tau). \]
Assume that $M$ is odd (this is the only case we need).
Then, its Jacobi transformation $\tau \mapsto -\frac{1}{\tau}$ is given as follows:
for $0<r<M$, 
\begin{equation}\label{Jacobi-F}
F_r^{(M)}\left(-\frac{1}{\tau}\right)=
2\left(\frac{\tau}{M\sqrt{-1}}\right)^{\frac{1}{2}}(-1)^{r -\frac{1}{2}(M+1)} 
\displaystyle{\sum_{\substack{0<r'<M \\ r' \equiv 0 \mod 2}}
\sin\left(\frac{rr'}{M}\pi\right)e^{\frac{1}{2}r'\pi\sqrt{-1}}F_{\frac{r'}{2}}^{(M)}(\tau)}.
\end{equation}

Notice that, for $0<r<M$, one has
\[ F_{M-r}^{(M)}(\tau)=F_r^{(M)}(\tau). \]
%--------------------------------------------------------
\subsection{Conclusion}
It turns out that the functions 
$\{\eta(\tau)^{l-1}F^{(2l+3)}_{l+1-i}(\tau)\}_{0\leq i\leq l}$ enjoys the same 
properties as $\{(-1)^{i}F^{i}(\tau)\}_{0\leq i\leq l}$, i.e., 
\eqref{Funct-Eq-F-i}. Thus, we see that
\[ J^{i}(\tau,z,t) \propto (-1)^{i} \eta(\tau)^{l-1}F_{l+1-i}^{(2l+3)}(\tau)
A_\rho(\tau,z,t). \]

In particular, since $\eta^{l-1}F_{l+1}^{(2l+3)}$ never vanishes on $\nh$, we see that, 
for any $\tau \in \nh$, the fundamental characters 
$\{\chi_{\Lambda_i}\}_{0\leq i\leq l}$ are algebraically independent. In particular, 
we have proved the validity of the first part of the Conjecture \ref{conj-KP}: 
\begin{thm}\label{thm_main1} For any $\tau \in \nh$, we have
\[ \left. \tTh^+\right\vert_{Y_\tau} 
=\C[\iota_{Y_{\tau}}^\ast(\chi_{\Lambda_i})\; 
(0\leq i\leq l)], \]
where $\iota_{Y_{\tau}}: Y_\tau \hookrightarrow Y$ is a 
natural embedding.
\end{thm}
Indeed, both of them are $\Z$-graded $\C$-algebras whose Poincar\'{e} series (cf. \eqref{Poincare-ser-Th^+}) are the same. 
\begin{rem} Let $L(c,h)$ be the irreducible highest weight module over the 
Virasoro algebra whose highest weight is $(c,h)$. For each integer $0\leq i\leq l$, 
set
\[ c_{2,2l+3}:=1-\frac{3(2l+1)^2}{2l+3}, \qquad 
h_{i}:=\frac{(2i+1)^2-(2l+1)^2}{8(2l+3)}.\]
We denote the normalized character $\tr_{L(c_{2,2l+3}, h_i)}(q^{L_0-\frac{1}{24}c})$ 
by $\chi_{1,l+1-i}(\tau)$. It turns out that $($cf. \cite{IK}$)$
\[ \chi_{1,l+1-i}(\tau)=\frac{F_{l+1-i}^{(2l+3)}(\tau)}{\eta(\tau)}. \]
In particular, this implies that
\[ J^{i}(\tau, z, t) \propto (-1)^{i} \eta(\tau)^l \chi_{1,l+1-i}(\tau)A_\rho(\tau,z,t). \]
\end{rem}

As for the second part of the Conjecture \ref{conj-KP}, we see that the 
$\cO_{\nh}$-algebra $\tTh^+$ contains the $\cO_\nh$-algebra generated by 
$\{\chi_{\Lambda_i}\}_{0\leq i\leq l}$, that is isomorphic to a polynomial algebra, 
as a subalgebra. In addition, as a graded $\cO_{\nh}$-algebra, the Poincar\'{e} 
series of $\tTh^+$ and $\cO_\nh[\chi_{\Lambda_i}\; (0\leq i\leq l)]$ coincide. 

Let $\cK_{\nh}$ be the quotient field of $\cO_\nh$. One has 
\[ \cK_{\nh}\otimes_{\cO_\nh} \tTh^+ \cong 
\cK_{\nh}[\chi_{\Lambda_i}\; (0\leq i\leq l)]; \qquad f \otimes F \; \longmapsto fF. \]
On the other hand, by Theorem \ref{thm_main1}, the above isomorphism restricts to
\[ \cO_{\nh, \tau} \otimes_{\cO_{\nh}} \tTh^+ 
\cong \cO_{\nh, \tau}[\chi_{\Lambda_i}\; (0\leq i\leq l)].
\]
Since $\bigcap_{\tau \in \nh} \cO_{\nh, \tau} =\cO_\nh$, this implies 
\begin{thm}\label{thm_main2} $\tTh^+ =\cO_\nh[\chi_{\Lambda_i}\; (0\leq i\leq l)].$
\end{thm}
Hence the second part of the Conjecture \ref{conj-KP} is also valid.

\end{document}